\documentclass[a4paper,12pt]{preprint}
\usepackage[margin=1.2in]{geometry}  
\usepackage[full]{textcomp}

\usepackage{mathalfa}
\usepackage{microtype}

\usepackage{booktabs}

\usepackage{enumitem}
\usepackage{amsmath}
\usepackage{amsfonts} 
\usepackage{amssymb}             

\usepackage{comment}
\usepackage{ulem}

\usepackage{mathrsfs}
\usepackage{booktabs}
\usepackage{mathtools}

\usepackage{tikz}
\usepackage{amsthm}                
\usepackage{mhequ}
\usepackage{hyperref}

\hypersetup{pdfstartview=}

\addtocontents{toc}{\protect\hypersetup{hidelinks}}

\DeclareSymbolFont{sfoperators}{OT1}{ptm}{m}{n}
\DeclareSymbolFontAlphabet{\mathsf}{sfoperators}

\makeatletter
\def\operator@font{\mathgroup\symsfoperators}
\makeatother

\numberwithin{equation}{section}

\newtheorem{thm}{Theorem}[section]
\newtheorem{defn}[thm]{Definition}
\newtheorem{lem}[thm]{Lemma}
\newtheorem{prop}[thm]{Proposition}

\newtheorem{assumption}[thm]{Assumption}
\theoremstyle{remark}

\newtheorem{rmk}[thm]{Remark}

\DeclareMathOperator{\id}{id}

\makeatletter
\def\th@newremark{\th@remark\thm@headfont{\bfseries}}

\def\bdiamond{\mathop{\mathpalette\bdi@mond\relax}}
\newcommand\bdi@mond[2]{%
	\vcenter{\hbox{\m@th
			\scalebox{\ifx#1\displaystyle 2.6\else1.8\fi}{$#1\diamond$}%
	}}%
}

\def\bDiamond{\mathop{\mathpalette\bDi@mond\relax}}
\newcommand\bDi@mond[2]{%
	\vcenter{\hbox{\m@th
			\scalebox{\ifx#1\displaystyle 2.6\else1.2\fi}{$#1\Diamond$}%
	}}%
}
\makeatletter

\definecolor{darkgreen}{rgb}{0.1,0.7,0.1}
\definecolor{darkred}{rgb}{0.7,0.1,0.1}
\definecolor{darkblue}{rgb}{0,0,0.7}
\newcommand{\aA}{\mathcal{A}}
\newcommand{\bB}{\mathcal{B}}
\newcommand{\cC}{\mathcal{C}}
\newcommand{\dD}{\mathcal{D}}

\newcommand{\fF}{\mathcal{F}}
\newcommand{\gG}{\mathcal{G}}
\newcommand{\hH}{\mathcal{H}}
\newcommand{\iI}{\mathcal{I}}

\newcommand{\oO}{\mathcal{O}}
\newcommand{\pP}{\mathcal{P}}
\newcommand{\qQ}{\mathcal{Q}}

\newcommand{\sS}{\mathcal{S}}
\newcommand{\tT}{\mathcal{T}}
\newcommand{\uU}{\mathcal{U}}

\newcommand{\fR}{\mathfrak{R}}
\newcommand{\fs}{\mathfrak{s}}

\newcommand{\fx}{\mathfrak{x}}
\newcommand{\fy}{\mathfrak{y}}

\makeatletter 
\newcommand{\cov}{{\operator@font cov}}
\newcommand{\var}{{\operator@font var}}
\newcommand{\corr}{{\operator@font corr}}
\newcommand{\diam}{{\operator@font diam}}
\newcommand{\Av}{{\operator@font Av}}
\newcommand{\trig}{{\operator@font trig}}
\newcommand{\Enh}{{\operator@font Enh}}
\newcommand{\EEnh}{\overline {\operator@font Enh}}
\makeatother

\newcommand{\KPZ}{\text{\tiny KPZ}}

\newcommand{\E}{\mathbf{E}}
\newcommand{\K}{\mathbf{K}}
\newcommand{\N}{\mathbf{N}}
\newcommand{\R}{\mathbf{R}}
\newcommand{\T}{\mathbf{T}}

\newcommand{\Z}{\mathbf{Z}}

\renewcommand{\k}{\mathbf{k}}
\renewcommand{\l}{\mathbf{\ell}}

\newcommand{\n}{\mathbf{n}}

\renewcommand{\r}{\mathbf{r}}
\renewcommand{\t}{\mathbf{t}}

\newcommand{\x}{\mathbf{x}}

\newcommand{\PPi}{\boldsymbol{\Pi}}

\newcommand*\Btheta{\ensuremath{\boldsymbol\theta}}

\def\set{{\mathfrak{u}}}

\newcommand{\hf}{\widehat{f}}
\newcommand{\hF}{\widehat{F}}

\newcommand{\hPi}{\widehat{\Pi}}

\def\Clus{\mathscr{C}}

\newcommand{\eps}{\varepsilon}

\colorlet{symbols}{blue}

\colorlet{testcolor}{green!60!black}

\def\${|\!|\!|}

\usetikzlibrary{shapes.misc}
\usetikzlibrary{shapes.symbols}
\usetikzlibrary{decorations}
\usetikzlibrary{decorations.markings}
\usetikzlibrary{decorations.pathmorphing}

\def\drawx{\draw[-,solid] (-3pt,-3pt) -- (3pt,3pt);\draw[-,solid] (-3pt,3pt) -- (3pt,-3pt);}
\tikzset{
	root/.style={circle,fill=testcolor,inner sep=0pt, minimum size=2mm},
	dot/.style={circle,fill=black,inner sep=0pt, minimum size=1mm},
	edot/.style={circle,fill=black,inner sep=0pt, minimum size=1mm},
	odot/.style={circle,draw=black,inner sep=0pt, minimum size=1mm},
	var/.style={circle,fill=black!10,draw=black,inner sep=0pt, minimum size=
	2mm},
    svar/.style={circle,fill=black!10,draw=black,inner sep=0pt, minimum size=
	1.5mm},
    noise0/.style={rectangle,draw=symbols,fill=white,inner sep=0pt, minimum size=1.5mm},
    noise1/.style={circle,draw=symbols,fill=white,inner sep=0pt, minimum size=1.5mm},
    noise2/.style={circle,draw=symbols,fill=symbols,inner sep=0pt, minimum size=1.5mm},
    noise3/.style={circle,draw=black,fill=black,inner sep=0pt, minimum size=1.5mm},
	dotred/.style={circle,fill=symbols!50,inner sep=0pt, minimum size=2mm},
	generic/.style={semithick,shorten >=1pt,shorten <=1pt},
	ageneric/.style={semithick},
	dist/.style={ultra thick,draw=testcolor,shorten >=1pt,shorten <=1pt},
	testfcn/.style={ultra thick,testcolor,shorten >=1pt,shorten <=1pt,<-},
	testfcnx/.style={ultra thick,testcolor,shorten >=1pt,shorten <=1pt,<-,
		postaction={decorate,decoration={markings,mark=at position 0.6 with {\drawx}}}},
	kepsilon/.style={semithick,shorten >=1pt,shorten <=1pt,densely dashed,->},
	kprimex/.style={semithick,shorten >=1pt,shorten <=1pt,densely dashed,->,
		postaction={decorate,decoration={markings,mark=at position 0.4 with {\drawx}}}},
	kernel/.style={semithick,shorten >=1pt,shorten <=1pt,->},
	akernel/.style={semithick,->},
	multx/.style={shorten >=1pt,shorten <=1pt,
		postaction={decorate,decoration={markings,mark=at position 0.5 with {\drawx}}}},
	kernelx/.style={semithick,shorten >=1pt,shorten <=1pt,->,
		postaction={decorate,decoration={markings,mark=at position 0.4 with {\drawx}}}},
	kernel1/.style={->,semithick,shorten >=1pt,shorten <=1pt,postaction={decorate,decoration={markings,mark=at position 0.45 with {\draw[-] (0,-0.1) -- (0,0.1);}}}},
	kernel2/.style={->,semithick,shorten >=1pt,shorten <=1pt,postaction={decorate,decoration={markings,mark=at position 0.45 with {\draw[-] (0.05,-0.1) -- (0.05,0.1);\draw[-] (-0.05,-0.1) -- (-0.05,0.1);}}}},
	kernelBig/.style={semithick,shorten >=1pt,shorten <=1pt,decorate, decoration={zigzag,amplitude=1.5pt,segment length = 3pt,pre length=2pt,post length=2pt}},
	gepsilon/.style={dotted,semithick,shorten >=1pt,shorten <=1pt},
	renorm/.style={shape=circle,fill=white,inner sep=1pt},
	labl/.style={shape=rectangle,fill=white,inner sep=1pt},
	xi/.style={circle,fill=symbols!10,draw=symbols,inner sep=0pt,minimum size=1.2mm},
	xix/.style={crosscircle,fill=symbols!10,draw=symbols,inner sep=0pt,minimum size=1.2mm},
	xib/.style={circle,fill=symbols!10,draw=symbols,inner sep=0pt,minimum size=1.6mm},
	xibx/.style={crosscircle,fill=symbols!10,draw=symbols,inner sep=0pt,minimum size=1.6mm},
	not/.style={circle,fill=symbols,draw=symbols,inner sep=0pt,minimum size=0.5mm},
	>=stealth,
  	highlight/.style={line width=7pt,blue,draw opacity=0.2,line cap=round,line join=round},
  	cover/.style={line width=7pt,blue,line cap=round,line join=round},
	smalldot/.style={circle,fill=symbols,draw=symbols, solid,inner sep=0pt,minimum size=0.5mm},
	}

\makeatletter
\def\DeclareSymbol#1#2#3{\expandafter\gdef\csname MH@symb@#1\endcsname{\tikz[baseline=#2,scale=0.15,draw=symbols]{#3}}\expandafter\gdef\csname MH@symb@#1s\endcsname{\scalebox{0.5}{\tikz[baseline=#2,scale=0.15,draw=symbols]{#3}}}}
\def\<#1>{\csname MH@symb@#1\endcsname}
\makeatother

\DeclareSymbol{Xi22}{0.5}{\draw (0,0) node[xi] {} -- (-1,1) node[not] {} -- (0,2) node[xi] {}; }

\DeclareSymbol{0'}{0}{\node at (0,0.7) [noise0] {}; }
\DeclareSymbol{1'}{0}{\node at (0,0.7) [noise1] {}; }
\DeclareSymbol{2'}{0}{\node at (0,0.7) [noise2] {}; }
\DeclareSymbol{3'}{0}{\node at (0,0.7) [noise3] {}; }

\DeclareSymbol{0'0}{0}{\draw (0,-0.5) node[not] {} -- (0,1.5) node[noise0] {};}
\DeclareSymbol{1'0}{0}{\draw (0,-0.5) node[not] {} -- (0,1.5) node[noise1] {};}
\DeclareSymbol{2'0}{0}{\draw (0,-0.5) node[not] {} -- (0,1.5) node[noise2] {};}
\DeclareSymbol{3'0}{0}{\draw (0,-0.5) node[not] {} -- (0,1.5) node[noise3] {};}
\DeclareSymbol{1'1'}{0}{\draw (-1,1.5) node[noise1] {}  -- (0.5,0) node[noise1] {};}
\DeclareSymbol{2'1'}{0}{\draw (-1,1.5) node[noise2] {}  -- (0.5,0) node[noise1] {};}
\DeclareSymbol{2'2'}{0}{\draw (-1,1.5) node[noise2] {}  -- (0.5,0) node[noise2] {};}
\DeclareSymbol{3'1'}{0}{\draw (-1,1.5) node[noise3] {}  -- (0.5,0) node[noise1] {};}
\DeclareSymbol{3'2'}{0}{\draw (-1,1.5) node[noise3] {}  -- (0.5,0) node[noise2] {};}
\DeclareSymbol{2'1'0}{0}{\draw (0,1.4) node[noise2] {}  -- (1.5,0.5) node[noise1] {} -- (0,-0.4) node[not] {};}
\DeclareSymbol{2'2'0}{-2}{\draw (-1,1) node[noise2] {}  -- (0,-0.5) node[smalldot] {} -- (1,1) node[noise2] {};}
\DeclareSymbol{2'1'1'}{-1}{\draw (0,1.4) node[noise2] {}  -- (1.5,0.5) node[noise1] {} -- (0,-0.4) node[noise1] {};}
\DeclareSymbol{2'0'}{0}{\draw (-1,1.5) node[noise2] {}  -- (0.5,0) node[noise0] {};}
\DeclareSymbol{2'2'0'}{0}{\draw (-1,1.5) node[noise2] {}  -- (0,0) node[noise0] {} -- (1,1.5) node[noise2] {};}

\DeclareSymbol{0}{0}{\draw[white] (-.4,0) -- (.4,0); \draw (0,0)  -- (0,1.2) node[smalldot] {};}
\DeclareSymbol{1}{0}{\draw[white] (-.4,0) -- (.4,0); \draw (0,0)  -- (0,1.2) node[smalldot] {};}
\DeclareSymbol{2}{0}{\draw (-0.5,1.2) node[smalldot] {} -- (0,0) -- (0.5,1.2) node[smalldot] {};}
\DeclareSymbol{11}{0}{\draw (0,1.8) node[smalldot] {} -- (-0.7,0.9) -- (0,0) -- (0.7,1) node[smalldot] {};}
\DeclareSymbol{10}{0}{\draw (0,1.8) node[smalldot] {} -- (-0.8,0.9) -- (0,0);}
\DeclareSymbol{21}{0.7}{\draw (-1,1.8) node[smalldot] {} -- (-0.5,0.9); \draw (0,1.8) node[smalldot] {} -- (-0.5,0.9) -- (0,0) -- (0.5,0.9) node[smalldot] {};}
\DeclareSymbol{20}{0.7}{\draw (-1,1.8) node[smalldot] {} -- (-0.5,0.9); \draw (0,1.8) node[smalldot] {} -- (-0.5,0.9) -- (-0.5,0);}
\DeclareSymbol{210}{1.1}{\draw (-0.5,2.4) node[smalldot] {} -- (-1,1.6);\draw (-1.5,2.4) node[smalldot] {} -- (-0.5,0.8); \draw (0,1.6) node[smalldot] {} -- (-0.5,0.8) -- (-0.5,0);}
\DeclareSymbol{211}{1.1}{\draw (-0.5,2.4) node[smalldot] {} -- (-1,1.6);\draw (-1.5,2.4) node[smalldot] {} -- (-0.5,0.8); \draw (0,1.6) node[smalldot] {} -- (-0.5,0.8) -- (0,0) -- (0.5,0.8) node[smalldot] {};}
\DeclareSymbol{22}{0.7}{\draw (-1.5,1.8) node[smalldot] {} -- (-1,0.9) -- (0,0) -- (1,0.9) -- (1.5,1.8) node[smalldot] {}; \draw (-0.5,1.8) node[smalldot] {} -- (-1,0.9);\draw (0.5,1.8) node[smalldot] {} -- (1,0.9);}

\setlist[itemize]{topsep=3pt,itemsep=1.5pt,parsep=0pt}

\def\scal#1{\langle#1\rangle}
\def\cent#1{\mathopen{{\langle\kern-0.3em\rangle}}#1\mathclose{{\langle\kern-0.3em\rangle}}}

\def\d{\partial}

\begin{document}

\title{A frequency-independent bound on trigonometric polynomials of Gaussians and applications}
\author{Fanhao Kong$^1$ and Wenhao Zhao$^2$}
\institute{Peking University, China, \email{fanhaokong@pku.edu.cn}
\and EPFL, Switzerland, \email{wenhao.zhao@epfl.ch}}

\maketitle

\begin{abstract}
	We prove a frequency-independent bound on trigonometric functions of a class of singular Gaussian random fields, which arise naturally from weak universality problems for singular stochastic PDEs. This enables us to reduce the regularity assumption on the nonlinearity of the microscopic models (for pathwise convergence) in KPZ and dynamical $\Phi^4_3$ in the previous works of Hairer-Xu and Furlan-Gubinelli to heuristically optimal thresholds required by PDE structures. 
\end{abstract}

\setcounter{tocdepth}{2}
\microtypesetup{protrusion=false}
\tableofcontents
\microtypesetup{protrusion=true}

\def\k{\mathbf{k}}


\section{Introduction}

The aim of this article is to prove a frequency-independent upper bound of moments for quantities of the form
\begin{equation}
    \label{e:main_form}
    \int \varphi^{\lambda}(x) \; \d_{\theta_\fx}^{r_1} \tT_{(m_1 - 1)} \big( \trig_{\pm}(\theta_{\fx}\eps^{\frac{\alpha}{2}} \Psi_\eps(x)) \big) \bigg( \int K(x,y) \; \d_{\theta_\fy}^{r_2} \tT_{(m_2 - 1)} \big( \trig_{\pm}(\theta_{\fy}\eps^{\frac{\alpha}{2}} \Psi_\eps(y) ) \big) dy \bigg) dx\;,
\end{equation}
where the integral is taken over $x,y \in \R^d$, $\{\Psi_\eps\}$ is a class of Gaussian random fields over $\R^d$ of ``singularity" $\frac{\alpha}{2}$, $\trig_{+}$ and $\trig_{-}$ denote cosine and sine functions respectively, $\theta_{\fx}$ and $\theta_{\fy}$ are the frequencies in the trigonometric functions, $\tT_{(m-1)}(\cdot)$ is the operator that removes the first $m-1$ chaos components from the random variable, $K$ is a suitable integration kernel with prescribed singularity at the origin, and $\varphi^\lambda$ is the smooth test function $\varphi$ with compact support rescaled at scale $\lambda \in (0,1)$. Quantities of the type \eqref{e:main_form} arise naturally from weak universality problems for singular stochastic PDEs, such as KPZ and dynamical $\Phi^4_3$.

\subsection{Motivation}
\label{sec:motivation}

The 1+1 dimensional KPZ equation is formally given by
\begin{equ} \label{e:kpz}
    \d_t h = \d_x^2 h + a (\d_x h)^2 + \xi, \quad (t,x) \in \R^{+} \times \T\;.
\end{equ}
Here, $\T$ is the one dimensional torus, $\xi$ is the space-time white noise, and $a \in \R$ denotes the coupling constant. The equation was first derived in \cite{KPZ86}, but is only formal since $\d_x h$ is distribution valued. The first rigorous statement for \eqref{e:kpz} came in \cite{BG97}, where the authors defined the solution via the Cole-Hopf transform from a linear equation with multiplicative noise. A different notion of solution, the energy solution, was introduced in \cite{Energy} and then shown to be unique in \cite{Energy_unique}. In \cite{HairerKPZ, rs_theory}, a pathwise solution theory was developed (see also \cite{GIP12,Reloaded}), where the solution to \eqref{e:kpz} is defined as the limit of the smooth solutions to the regularised and renormalised equation
\begin{equ}
    \d_t h_{\eps} = \d_x^2 h_{\eps} + a (\d_x h_{\eps})^2 + \xi_{\eps} - C_{\eps}, 
\end{equ}
where $\xi_\eps$ is a smooth approximation to $\xi$ at scale $\eps$, and $C_{\eps}=\frac{c}{\eps}+\oO(1)$ is the renormalisation constant that ensures the convergence of $h_\eps$ to a nontrivial limit. In fact, for every $a$, there is one degree of freedom in choosing the limit by the choice of the $\oO(1)$ counterterm. Hence, there is a one-dimensional family which are essentially the same up to the $\oO(1)$ counterterm. We denote this family by KPZ$(a)$.

The solution to the KPZ equation is expected to describe the universal large scale behaviour for a wide class of weakly asymmetric interface growth models. In \cite{HQ}, the authors considered microscopic growth models of the type
\begin{equ} \label{e:micro}
    \d_t \widetilde{h} = \d_x^2 \widetilde{h} + \sqrt{\eps} F(\d_x \widetilde{h}) + \widetilde{\xi}\,, \quad (t,x)\in\R\times(\T/\eps),
\end{equ}
where $F$ is an even polynomial and $\widetilde{\xi}$ is a smooth space-time Gaussian field with short range correlation on $\R\times(\T/\eps)$. They showed that there exists $C_\eps \rightarrow +\infty$ such that the rescaled and re-centered process
\begin{equ} 
\label{e:recentred_process}
h_{\eps}(t,x):= \sqrt{\eps} \ \widetilde{h} (t/\eps^2, x / \eps) - C_{\eps} t
\end{equ}
converges to the KPZ($a$) family of solutions, where the coupling constant $a$ depends on all the coefficients of $F$ except the constant term. 

Let us briefly explain the Hairer-Quastel universality result. The macroscopic process $h_\eps$ defined above satisfies the equation
\begin{equ} \label{e:kpz_macro}
\d_t h_{\eps} = \d_x^2 h_{\eps} +\eps^{-1} F(\sqrt{\eps} \d_x h_{\eps}) + \xi_{\eps} - C_{\eps},
\end{equ}
where $\xi_{\eps}(t,x) = \eps^{-\frac{3}{2}} \widetilde{\xi}(\frac{t}{\eps^2},\frac{x}{\eps})$ is a scale-$\eps$ approximation to $\xi$. Let $Z_{\eps}$ be the solution to the linear part of \eqref{e:kpz_macro} (that is, with $F$ and $C_\eps$ removed), and $\Psi_{\eps}:= \d_x Z_{\eps}$. Then the remainder $u_{\eps} := h_{\eps} - Z_{\eps}$ satisfies the equation
\begin{equ} \label{e:21_1}
	\d_{t}u_{\eps} = \d_x^2 u_{\eps} + \eps^{-1} F(\sqrt{\eps} \Psi_{\eps} + \sqrt{\eps}\d_{x}u_{\eps}) - C_{\eps}.
\end{equ}
Since $\sqrt{\eps} \Psi_{\eps}$ is asymptotically normal distributed and $\sqrt{\eps}\d_x  u_{\eps}$ is expected to have size $\eps^{\frac{1}{2}-}$, if $F$ has more than two derivatives, it is natural to Taylor expand $F$ near $\sqrt{\eps} \Psi_{\eps}$ to the second order to get
\begin{equation*}
	\begin{split}
		\eps^{-1}F(\sqrt{\eps} \Psi_{\eps} + \sqrt{\eps} \d_{x}u_{\eps}) = &\eps^{-1}F(\sqrt{\eps} \Psi_{\eps}) + \eps^{-\frac{1}{2}} F'(\sqrt{\eps} \psi_{\eps}) \, \d_{x}u_{\eps}\\
		&+ \frac{1}{2} F''(\sqrt{\eps} \Psi_{\eps}) \, ({\d_{x}u_{\eps}})^{2} + o_{\eps}(1)\;.
	\end{split}
\end{equation*}
If one shows the quantities $\eps^{-1}F(\sqrt{\eps}\Psi_{\eps}) - C_{\eps}$, $\eps^{-\frac{1}{2}}F'(\sqrt{\eps} \Psi_{\eps})$, $F''(\sqrt{\eps} \Psi_{\eps})$, as well as their certain products and with heat kernel convolutions (coming from multiplication of $\d_x u_\eps$) converge to the right functionals of Gaussian random fields, then one can establish the (pathwise) convergence of \eqref{e:kpz_macro} to KPZ($a$) for general $F$. This was first established in \cite{HQ} for even polynomials $F$ and later extended to $F \in \cC^{7+}$ in \cite{HX19}. See also \cite{KPZCLT, Chatterjee_kpz} for related models with non-Gaussian approximations to the space-time white noise. 

With the notion of energy solution, convergence in law of the process $h_\eps$ to the KPZ($a$) family was obtained in \cite{HQ_stationary} and \cite{HQ_non_stationary} when $F$ is only Lipschitiz.\footnote{Note that although \cite{HQ_stationary} and \cite{HQ_non_stationary} only assumes $F$ being Lipschitz, the results in \cite{HQ, HX19} and Theorem~\ref{thm:kpz} below are not included in it since the notions of convergence are different. Moreover, the notion of energy solution is available for the KPZ equation but not for $\Phi^4_3$ so far.} 

Similar weak universality results are also investigated for the dynamical $\Phi^4_3$ model on the three dimensional torus, formally given by
\begin{equation} \label{e:phi43}
    \d_t \phi = \Delta \phi - a \phi^3 + \xi\;, \qquad (t,x) \in \R^+ \times \T^3\;.
\end{equation}
Similar to \cite{HQ}, one can consider approximations to \eqref{e:phi43} via general phase-coexistence models of the type
\begin{equation} \label{e:phi43_macro}
    \d_t \phi_\eps = \Delta \phi_\eps - \eps^{-\frac{3}{2}} G(\sqrt{\eps} \phi_\eps) + \xi_\eps + C_\eps \phi_\eps\;,
\end{equation}
where $G$ is a nice odd function, $\xi_\eps$ is the scale-$\eps$ approximation to $\xi$ as previous, and $C_\eps$ is a renormalisation constant. Let $\Psi_\eps$ denote the solution to the linearised equation. Then $u_\eps := \phi_\eps - \Psi_\eps$ satisfies
\begin{equation*}
    \d_t u_\eps = \Delta u_\eps - \eps^{-\frac{3}{2}} G(\sqrt{\eps} \Psi_\eps + \sqrt{\eps} u_\eps) + C_\eps (\Psi_\eps + u_\eps)\;.
\end{equation*}
Similar as before, one expects that $\sqrt{\eps} \Psi_\eps$ is asymptotically normal distributed, and $\|\sqrt{\eps} u_\eps\|_{L^\infty} = \oO(\eps^{\frac{1}{2}-})$. Hence, if $G$ has more than three derivatives, we can Taylor expand $G$ at $\sqrt{\eps} \Psi_\eps$ up to the third order to get
\begin{equation} \label{e:phi43_general}
    \begin{split}
    \eps^{-\frac{3}{2}} G(\sqrt{\eps} \Psi_\eps + &\sqrt{\eps} u_\eps) =  \eps^{-\frac{3}{2}} G(\sqrt{\eps} \Psi_\eps) + \eps^{-1} G'(\sqrt{\eps} \Psi_\eps) \cdot u_\eps\\
    &+ \frac{1}{2 \sqrt{\eps}} G''(\sqrt{\eps} \Psi_\eps) \cdot u_\eps^2 + \frac{1}{6} G^{(3)}(\sqrt{\eps} \Psi_\eps) \cdot u_\eps^3 + o_\eps(1)\;.
    \end{split}
\end{equation}
If one can establish convergence of the above functionals of $\sqrt{\eps} \Psi_\eps$ (with suitable renormalisation) as well as their products (also with heat kernel convolutions) to the correct functionals of the Gaussians, then this implies the convergence of the process $\phi_\eps$ in \eqref{e:phi43_macro} to the dynamical $\Phi^4_3(a)$ model with $a$ depending on \textit{all details of $G$}. This was done in \cite{Phi4_poly} for $G$ odd polynomial and extended in \cite{Phi4_general} to $G \in \cC^{9+}$. 

Let us briefly explain the difference between polynomial and non-polynomial nonlinearities in \eqref{e:kpz_macro} and \eqref{e:phi43_macro}. If $F$ and $G$ are polynomials, convergence of the stochastic objects follows from direct chaos expansions of $F$ and $G$. However, for general non-polynomial functions, showing that these stochastic objects converge are nontrivial even for analytic functions. The main problem is that the chaos series is infinite, and it turns out that high moments of each term in the series is not summable unless the Fourier transforms of the nonlinearities decay faster than Gaussians. This problem was resolved in \cite{HX19} and \cite{Phi4_general} with different methods. In \cite{HX19}, the authors Fourier expand the nonlinearity, and develop a clustering method to control trigonometric functions of the Gaussian fields. This allows them to obtain a bound in the desired regularity space with polynomial dependence on the frequency, which in turn shows the convergence to KPZ for $F \in \cC^{7+}$. For the dynamical $\Phi_3^4$ model, a similar universality result was shown to hold for $G \in \cC^{9+}$ in \cite{Phi4_general}, where the authors developed a Malliavin calculus based method to control the stochastic objects. 

On the other hand, if $F$ is below $\cC^{2}$ or $G$ is below $\cC^3$, then it is not clear how to proceed with the heuristic Taylor expansions as above. In \cite{KPZ_Deterministic}, twice differentiability of the nonlinearity is also needed to get a scaling limit for the deterministic KPZ model. One expects that new ideas are needed to go below the threshold for pathwise convergence.

In this article, we further explore the method developed in \cite{HX19}, and prove a \textit{frequency independent} bound for an integral version of correlations for trigonometric polynomials of Gaussians with two frequencies. This extends the frequency-independent bound in \cite{Xu18} for the one-frequency case, and turns out to be sufficient to reduce the regularity requirement to $F \in \cC^{2+}$ for KPZ and $G \in \cC^{3+}$ for dynamical $\Phi^4_3$. We will explain the main obstacles and our idea to overcome it in Section~\ref{sec:proof}. We hope that the bound and the method of proving it might be useful in other situations as well. 

Finally, let us mention that recent developed methods for stochastic estimates via spectral gap (\cite{diagram_free_rs, bphz_spectral_gap}) are also likely to be used to treat the above situations. 

\subsection{Statement of the main bound}
\label{sec:statement}

Fix a scaling $\mathfrak{s} = (s_1, ..., s_d)$ on $\R^d$ and set $|\fs|=s_1+\dots+s_d$. The metric induced by $\mathfrak{s}$ is $|x|_{\mathfrak{s}} := \sup_{1\leq i \leq d} |x_i|^{\frac{1}{s_i}}$. Since the scaling is fixed, we simply write $|x|$ instead of $|x|_{\mathfrak{s}}$. To avoid mixing up the notations $|\fs|$ and $|x|$, we point out that all the notations $|\cdot|$ with respect to the variables in $\R^d$ except $|\fs|$ represent the metric induced by $\fs$. Let $\{\Psi_{\eps}\}_{\eps\in[0,1]}$ be a class of Gaussian random fields satisfying the following assumption.

\begin{assumption} \label{as:cor}
    $\Psi_{\eps}$ is centered and stationary Gaussian random field, and there exist $\alpha \in (0,|\fs|)$ and $\Lambda > 1$ such that
    \begin{equ}
        \frac{1}{\Lambda(|x-y|+\eps)^\alpha}\leq\E(\Psi_{\eps}(x)\Psi_{\eps}(y))\leq\frac{\Lambda}{(|x-y|+\eps)^\alpha}
    \end{equ}
    for all $x,y\in\R^d$ and all $\eps \in (0,1)$.
\end{assumption}
\begin{rmk}
    In application for the KPZ equation, $\Psi_{\eps} = \d_x P \ast \xi_{\eps}$, where $P$ is the heat kernel, $\ast$ is space-time convolution and $\xi_{\eps}$ is the mollified one-dimensional space-time white noise. For dynamical $\Phi^4_3$, it is taken to be $P\ast \xi_{\eps}$, where $\xi_{\eps}$ is the mollified three-dimensional space-time white noise. It is not hard to check that in both the cases the assumption is satisfied (See \cite[Proposition~4.1]{HX19} for example).
\end{rmk}
We use $X = \eps^{\frac{\alpha}{2}} \Psi_{\eps}(x)$ and $Y = \eps^{\frac{\alpha}{2}} \Psi_{\eps}(y)$ to denote the corresponding Gaussian random variable. For any centered Gaussian random variable $Z$ and $F: \R \rightarrow \R$ such that $F(Z)$ has finite second moment, let $C_n := \E F^{(n)}(Z) / n!$ be the coefficient of the $n$-th term in the chaos expansion of $F(Z)$, and
\begin{equation*}
    \tT_{(m)} \big( F(Z) \big) := F(Z) - \sum_{n \leq m} C_n Z^{\diamond n}
\end{equation*}
denote the random variable $F(Z)$ with the first $m$ chaos components removed. Here, $Z^{\diamond n}:=(\var Z )^{\frac{n}{2}} H_n(Z/\sqrt{\var Z})$ denotes the $n$-th Wick power of $Z$ where $H_n$ is the $n$-th Hermite polynomial.
For $\Btheta = (\theta_{\fx}, \theta_{\fy}) \in \R^2$, let
\begin{equ} \label{e:ff}
\fF(\Btheta,x,y) = \tT_{(m_1-1)} \big(\trig_{\zeta_1}(\theta_{\fx} X) \big) \; \tT_{(m_2-1)} \big(\trig_{\zeta_2}(\theta_{\fy} Y) \big),
\end{equ}
where $\zeta_j \in \{+,-\}$, $\trig_{+} = \cos$ and $\trig_{-} = \sin$. We take the convention that $m_j$ is odd when $\trig_{\zeta_j} = \sin$, and even when $\trig_{\zeta_j} = \cos$. For every $\lambda > 0$, every smooth $\varphi: \R^d \rightarrow \R$ and every $x \in \R^d$, let
\begin{equation*}
    \varphi^\lambda_x(y) := \lambda^{-|\fs|} \varphi\Big( \frac{y_1-x_1}{\lambda^{s_1}}, \dots, \frac{y_d - x_d}{\lambda^{s_d}} \Big)
\end{equation*}
be $\varphi$ centered at $x$ and rescaled at scale $\lambda$. We also write $\varphi^\lambda = \varphi_0^\lambda$. 

Let $\gamma > 0$ and $K_0: \R^d \rightarrow \R^+$ be compactly supported in the ball of radius $1$ containing the origin such that
\begin{equ} \label{e:kernel_assumption}
    \|K_0\|_{|\fs|-\gamma; p} := \sup_{|x|\leq 1, |k|<p} |x|^{|\fs|-\gamma + |k|}|D^kK_0(x)| < \infty.
\end{equ}
for every $p>0$. Let
\begin{equation*}
    r_e := \lceil \gamma-\frac{\alpha m_2}{2} \rceil \vee 0\;.
\end{equation*}
be the degree of positive renormalisation. Define the kernel
\begin{equ} \label{e:kernel}
    K(x, y) := K_0(x-y) - \sum_{|j|< r_e}\frac{x^j}{j!}D^jK_0(-y)\;.
\end{equ}
For $\eps, \lambda \in (0,1)$, also define the operator $\aA_{\eps,\lambda}$ by
\begin{equ} \label{e:operator}
(\aA_{\eps,\lambda} \fF)(\Btheta) = \int \varphi^{\lambda}(x) K(x,y) \fF(\Btheta,x,y) dx dy\;.
\end{equ}
Note that since $K_0$ is compactly supported, the integral in \eqref{e:operator} is taken in a finite region. So we suppose $|y|\leq 2$ throughout this paper without loss of generality.

\begin{assumption} \label{as:stochastic}
    Throughout this article, we assume the following constraints on the parameters $\alpha$, $m_1$, $m_2$, $\gamma$ and $|\fs|$:
    \begin{itemize}
        \item $0<\alpha m_1,\alpha m_2< |\fs|$\;, 
        \item $\alpha(m_1+m_2)\leq |\fs| + 2\gamma$\;, 
        \item $0<\gamma \leq \frac{|\fs|}{2}$\;.
    \end{itemize}
\end{assumption}

\begin{rmk}
    The first two assumptions are natural scaling constraints from \cite[Appendix]{HQ}. The additional technical assumption for $\gamma$ is used in the proof of Lemmas~\ref{le:region_smallness_int1} and~\ref{le:region_smallness_int2}. All relevant objects in KPZ and $\Phi^4_3$ satisfy all three constraints. 
\end{rmk}

We use the notation $\|\cdot\|_{2n} = (\E|\cdot|^{2n})^{\frac{1}{2n}}$ and use the relation $A\lesssim B$ to represent that there exists a constant $C$ (independent of some parameters) such that $A\leq CB$. The following is our main theorem.

\begin{thm}\label{thm:main_result}
	Suppose the Gaussian field $\Psi_{\eps}(x)$ and the constants used to define $\aA_{\eps,\lambda}$ satisfy the assumptions above. Then for every small $\eta > 0$, every $n\in \N$ and every $\r \in \N^2$, we have
	\begin{equ}
	\label{eq:main_result}
		\|\partial_{\Btheta}^{\r}(\aA_{\eps,\lambda} \fF)(\Btheta)\|_{2n} \lesssim \eps^{\frac{\alpha(m_1+m_2)}{2}-\frac{\alpha(m_1+m_2)}{2n}-\frac{\eta}{2}} \lambda^{-\frac{\alpha(m_1+m_2)}{2}+\gamma-\eta},
	\end{equ}
    uniformly over $\eps, \lambda \in (0,1)$, $\Btheta \in \R^2$ and smooth test functions $\varphi$ supported in the unit ball containing the origin with $\|\varphi\|_{L^\infty} \leq 1$, where $\d_{\Btheta}^{\r} = \d_{\theta_\fx}^{r_1} \d_{\theta_\fy}^{r_2}$. As a consequence, we have
	\begin{equ}
	\label{eq:main_result'}
		\|\partial_{\Btheta}^{\boldsymbol{r}}(\aA_{\eps,\lambda} \fF)(\Btheta)\|_{2n} \lesssim \eps^{\frac{\alpha(m_1+m_2)}{2}-\eta}\lambda^{-\frac{\alpha(m_1+m_2)}{2}+\gamma-\eta}.
	\end{equ}
\end{thm}

\begin{rmk}
    The bound \eqref{eq:main_result'} follows from \eqref{eq:main_result} by taking $n$ sufficiently large (but independent of $\eps$ and $\lambda$) and applying Jensen inequality. \eqref{eq:main_result} may seem confusing at first sight since the left hand side increases as $n$ gets larger while the right hand side decreases. Our point is that \eqref{eq:main_result} should be viewed as both sides are raised to their $2n$-th powers, and then a factor $\eps^{-\alpha(m_1+m_2)}$ is lost for some technical reasons in Lemma~\ref{le:fixed_bound}, but its exponent is independent of $n$. Actually, we treat \eqref{eq:main_result} as a technical intermediate step of proving \eqref{eq:main_result'}. 
\end{rmk}

\subsection{Applications to weak universality problems}


In Section~\ref{sec:spde} below, we will apply Theorem~\ref{thm:main_result} to KPZ and dynamical $\Phi^4_3$ models to reduce regularity requirements of the nonlinearity to the heuristic level discussed above. In this section, we first specify our requirements for the nonlinearity, and give statements on KPZ and $\Phi^4_3$. 

\begin{defn} \label{def:nonlinearity_space}
    For $k \in \N$ and $\beta \in [0,1)$, the class $\cC_{\pP}^{k,\beta}(\R, \R)$ consists of functions $F: \R \rightarrow \R$ such that there exists $C, M>0$ such that
    \begin{equation*}
        \sup_{0 \leq \ell \leq k} |F^{(\ell)}(u)| \leq C (1 + |u|)^{M}\;, \qquad \sup_{|h|<1} \frac{|F^{(k)}(u+h) - F^{(k)}(u)|}{|h|^\beta} \leq C(1 + |u|)^{M}
    \end{equation*}
    for all $u \in \R$. We denote it by $\cC_{\pP}^{k,\beta}$ for simplicity. 
\end{defn}

In what follows, we take $\xi_\eps$ in \eqref{e:kpz_macro} and \eqref{e:phi43_macro} to be the space-time mollification of the space-time white noise $\xi$ such that
\begin{equation*}
    \xi_\eps \overset{law}{=} \xi * \rho_\eps\;,
\end{equation*}
where $\rho$ is a space-time mollifier, and $\rho_\eps(x) := \eps^{-|\fs|} \rho(x_1/\eps^{s_1}, \dots, x_d / \eps^{s_d})$. Here $x$ denotes the $d$-dimensional space-time point, not just space. The space-time dimension and (parabolic) scaling in KPZ and $\Phi^4_3$ are $d=2, \fs = (2,1)$ and $d=4, \fs = (2,1,1,1)$ respectively. In the case of KPZ, we further assume $\rho$ is symmetric in the space variable $x_2$.

For $\gamma\in (1,2), \eta\in (0,1)$ and $\eps \in (0,1)$, we define the norm $\|\cdot\|_{\gamma,\eta;\eps}$ as 
\begin{equ}
    \|u\|_{\gamma,\eta;\eps} := \|u\|_{\cC^{\eta}} + \frac{\|Du\|_{L^{\infty}}}{\eps^{\eta-1}} + \|Du\|_{L^{\infty}} + \sup_{\substack{|x-y|\leq \eps\\x\neq y} }\frac{|Du(x)-Du(y)|}{\eps^{\eta - \gamma}|x-y|^{\gamma-1}}\;.
\end{equ}
For $\gamma\in (1,2), \eta<0$ and $\eps \in (0,1)$, we define the norm $\|\cdot\|_{\gamma,\eta;\eps}$ as 
\begin{equ}
    \|u\|_{\gamma,\eta;\eps} := \|u\|_{\cC^{\eta}} + \frac{\|u\|_{L^{\infty}}}{\eps^{\eta}} + \sup_{\substack{|x-y|\leq \eps\\x\neq y} }\frac{|Du(x)-Du(y)|}{\eps^{\eta - \gamma}|x-y|^{\gamma-1}}\;.
\end{equ}
One should think of it as a norm describing $\cC^{\gamma}$-norm at smaller scales (smaller than $\eps$) and $\cC^{\eta}$-norm at larger scales. We denote the space of functions with finite $\|\cdot\|_{\gamma,\eta;\eps}$ norm by $\cC^{\gamma,\eta}_\eps$. 

The following theorem is our result for the KPZ equation. 

\begin{thm} \label{thm:kpz}
    Let $\beta \in (0,1)$ and $F \in \cC_{\pP}^{2,\beta}$ be an even function with growth power $M$ (as in Definition \ref{def:nonlinearity_space}). Suppose $h_\eps(0, \cdot)\in\cC^{\gamma,\eta}_\eps$ is a class of functions on $\T$ such that there exists $h(0, \cdot) \in \cC^{\eta}$ such that $\|h_\eps(0, \cdot), h(0, \cdot)\|_{\gamma,\eta;\eps} \rightarrow 0$ in the sense of \cite[eq.(3.6)]{HX19} for some $\gamma \in (\frac{3}{2}, \frac{5}{3})$ and $\eta \in (\frac{1}{2}-\frac{1}{M+4}, \frac{1}{2})$\footnote{The notion of convergence of the initial data is roughly ``$\cC^{\eta}$ at large scales and $\cC^\gamma$ at small scales". The precise notion requires introduction of weighted spaces to overcome the non-integrable singularity at $t =0$. Since these have been treated in \cite{HQ, HX19} and does not require any modifications here, so we omit it for conciseness of the article and refer the readers to the references.}. Then there exists $C_\eps\rightarrow +\infty$ such that for every $\kappa>0$, the process $h_\eps$ in \eqref{e:kpz_macro} with initial data $h_\eps(0,\cdot)$ converges in probability in $C^{\eta}([0,1] \times \T)$ to the KPZ($a$) family with initial data $h(0, \cdot)$, where
    \begin{equation} \label{e:a_1}
        a = \frac{1}{2} \int_{\R} F''(x-y) \, \mu(dy)\;,
    \end{equation}
    and $\mu$ is the law of the Gaussian variable $(\d_x P * \widetilde{\xi})(0)$, where $\widetilde{\xi}$ is the smooth Gaussian field in microscopic model \eqref{e:micro} viewed on the whole space.\footnote{Although $\widetilde{\xi}$ depends on $\eps$, it can be viewed on the whole space as $\eps\rightarrow0$. The same convention is used in Theorem~\ref{thm:phi43}.}
\end{thm}

We have a similar result for the dynamical $\Phi^4_3$ equation. 

\begin{thm} \label{thm:phi43}
    Let $\beta \in (0,1)$, and $G \in \cC_{\pP}^{3,\beta}$ be an odd function. Let $M$ be the growth exponent for $G$ as in Definition~\ref{def:nonlinearity_space}. Suppose $\phi_\eps(0, \cdot)\in\cC^{\gamma,\eta}_\eps$ is a class of functions on $\T^3$ such that $\|\phi_\eps(0,\cdot),\phi(0,\cdot) \|_{\gamma,\eta;\eps} \rightarrow 0$ for some $\gamma \in (1,\frac{6}{5})$ and $\eta \in (-\frac{1}{2} - \frac{1}{2M}, -\frac{1}{2})$ in the sense of \cite[Definition~3.3]{Phi4_poly}, then there exists $C_\eps \rightarrow +\infty$ such that for every $\kappa>0$, the processes $\phi_\eps$ in \eqref{e:phi43_macro} with initial data $\phi_\eps(0,\cdot)$ converges in probability in $\cC^{-\frac{1}{2}-\kappa}([0,1] \times \T^3)$ to the dynamical $\Phi^4_3(a)$ family with initial data $\phi(0,\cdot)$, where
    \begin{equation} \label{e:a_2}
        a = \frac{1}{6} \int_{\R} G^{(3)}(x-y) \mu (dy)\;,
    \end{equation}
    and $\mu$ is the law of the Gaussian variable $(P * \widetilde{\xi})(0)$, where $\widetilde{\xi}$ is the smooth Gaussian field in the corresponding microscopic model for $\Phi^4_3$ viewed on the whole space. 
\end{thm}

\begin{rmk}
    The bound \eqref{eq:main_result} only involves Wick renormalisation, but both KPZ and $\Phi^4_3$ have stochastic objects that require renormalisations beyond Wick ordering. In fact, \eqref{eq:main_result} is used in combination with a trick from \cite{HX19} that splits every stochastic object into a large ``nice" part and a small part, and also with \cite[Theorem~6.2]{HX19} that shows convergence of the large ``nice" part to the desired limit with minimal assumption on the nonlinearity. Thanks to the fact that all second order divergences (beyond Wick) in KPZ and $\Phi^4_3$ are only \textit{logarithmic} in $\eps$, the smallness in the remainder is enough to kill the logarithmic divergence. This is why one can ignore second order renormalisations for the remainder, and apply Theorem~\ref{thm:main_result} to it to reduce the regularity requirement for the nonlinearity. 
\end{rmk}

\begin{rmk}
    Another subtlety arising from the KPZ equation is that there is an object with three appearances of the nonlinearity $F$ (and its derivative), and hence with three frequencies. So it is beyond the scope of Theorem~\ref{thm:main_result}. Here, we made use of the smallness of the ``remainder term" for this stochastic object in a way that reduces the remainder term to an analytically well-defined product between two simpler stochastic objects (with one and two frequencies). This will be done in detail in Section~\ref{sec:reg_0}. 
\end{rmk}


\subsection*{Notations}

For $p \geq 1$, we use $\|\cdot\|_{p}$ to denote the norm $\big( \E |\cdot|^{p} \big)^{\frac{1}{p}}$. For $\alpha<0$, we define the negative H\"older norm $\|\cdot\|_{\cC^{\alpha}}$ on distributions by
\begin{equation*}
    \|f\|_{\cC^{\alpha}} := \sup_{z\in\R^{d}} \sup_{\lambda\in(0,1)} \sup_{\varphi} \lambda^{-\alpha} |\scal{f,\varphi^{\lambda}_{z}}|,
\end{equation*}
where the test function $\varphi$ is taken over all functions supported in the unit ball with $\|\varphi\|_{\cC^{\lceil-\alpha\rceil}} \leq 1$. 

The relation $A\lesssim B$ implies that there exists a constant $C$ (independent of some parameters) such that $A\leq CB$. The relation $A\lesssim_n B$ implies that the proportionality constant $C$ depends on $n$. The relation $\gtrsim$ and $\gtrsim_n$ are defined similarly.


We denote the Wick product between Gaussian random variables with $\diamond$, for example $X^{\diamond a} \diamond Y^{\diamond b}$. For Gaussian random variable $Z$, function $F: \R \rightarrow \R$ such that $\E |F(Z)|^2 < +\infty$ and $m \in \N$, we define
\begin{equation*}
    \tT_{(m-1)} \big( F(Z) \big) := \sum_{n \geq m} C_n Z^{\diamond n}
\end{equation*}
to be the random variable $F(Z)$ with the first $m-1$ Wiener chaos component removed, and $C_n := \E F^{(n)}(Z) / n!$ is the coefficient of the $n$-th term in the chaos expansion. 

Finally, we define the Fourier transform $\hf$ of a function $f$ by
\begin{equation*}
    f(x) = \int \hf(\theta) e^{i\theta x} dx\;.
\end{equation*}

\subsection*{Organisation of the article}
This paper is organised as follows. We prove our main bound Theorem~\ref{thm:main_result} in Section~\ref{sec:proof}. In Section~\ref{sec:spde}, we apply the bound to weak universality problems for KPZ and dynamical $\Phi^4_3$ equations, and show that Theorem~\ref{thm:main_result} enables us to reduce the assumptions on the nonlinearities in \cite[Assumption~1.1]{HX19} and \cite[Assumption~1]{Phi4_general}. Finally in Appendix~\ref{sec:singleton_proof}, we prove some pointwise correlation bounds stated in Section~\ref{sec:correlation_bounds} and used in the proof of Theorem~\ref{thm:main_result}. 

\subsection*{Acknowledgments}
We thank Weijun Xu for suggesting the question and helpful discussions. We are also grateful to an anonymous referee for useful comments and suggestions on a previous draft of this work. Wenhao Zhao is partially supported by the elite undergraduate training program of School of Mathematical Sciences in Peking University.

\section{Proof of Theorem~\ref{thm:main_result}}
\label{sec:proof}

It suffices to prove the bound \eqref{eq:main_result}. We fix $\eps, \lambda \in (0,1)$. All proportionality constants below will be independent of $\eps$ and $\lambda$. Fix $n \in \N$, and write
\begin{equation*}
    \vec{x} = (x_1, \dots, x_{2n}) \in (\R^d)^{2n}\;, \qquad \vec{y} = (y_1, \dots, y_{2n}) \in (\R^d)^{2n}\;.
\end{equation*}
For $j = 1, \dots, 2n$, we write $X_j = \eps^{\frac{\alpha}{2}} \Psi_\eps(x_j)$ and $Y_j = \eps^{\frac{\alpha}{2}} \Psi_\eps(y_j)$. We also write $X,Y$ for $\eps^{\frac{\alpha}{2}}\Psi_{\eps}(x)$ and $\eps^{\frac{\alpha}{2}}\Psi_{\eps}(y)$. These quantities all depend on $\eps$, but since our bounds will be independent of $\eps$, we omit its dependence in notation for simplicity. 

With these notations, we have
\begin{equation} \label{e:main_expression}
    \begin{split}
    \|\big(\d_{\Btheta}^{\r} \aA_{\eps, \lambda} \fF\big)(\Btheta)\|_{2n}^{2n} = &\int \cdots \int \Big( \prod_{i=1}^{2n} \varphi^\lambda(x_i) \Big) \; \Big( \prod_{i=1}^{2n} K(x_i, y_i) \Big)\\
    &\E \prod_{i=1}^{2n} \Big( \d_{\theta_\fx}^{r_1}\tT_{(m_1-1)}\big( \trig_{\zeta_1} (\theta_\fx X_i) \big) \, \d_{\theta_\fy}^{r_2} \tT_{(m_2-1)} \big( \trig_{\zeta_2} (\theta_\fy Y_i) \big) \Big) d \vec{y} d \vec{x}\;.
    \end{split}
\end{equation}
Note that $\varphi$ is taken to be a test function with compact support in the unit ball $B(0,1)$, so we have $\operatorname{supp}(\varphi^{\lambda}) \subset B(0,\lambda)$. Also, our assumption for $K$ implies that $K(x_i,y_i) = 0$ if $|x_i-y_i| > 1$. Therefore, we can restrict the domain of the integral to $|x_i|,|y_i| \leq 2$. Moreover, we partition such integration domain into different parts, and will treat them differently. Let $L = 3n L_0$, where $L_0$ is the large constant defined at the beginning of Appendix~\ref{sec:singleton_proof}. Its value depends on $n$, $\r$ and $\Lambda$ but is independent of $\eps$, $\lambda$ and $\Btheta$. Let $\sS_{2n}$ be the configuration of $2n$ space-time points in $\R^d$ such that
\begin{equation*}
    \sS_{2n} := \Big\{ \vec{z} = (z_1, \dots, z_{2n}): \; \exists \; i \; \text{such that} \; |z_i - z_j| > L \eps \; \text{for all} \; j \neq i \Big\}\;.
\end{equation*}
For each $m \geq 2$, let
\begin{equation*}
\begin{split}
    \cC_m := \Big\{ &(z_1, \dots, z_{m}): \; \text{For every}\; 1\leq j_1<j_2\leq m,\;\exists \; I=\{i_1,\dots,i_\ell\}\subset[m]\;\\&\text{such that}\; |z_{i_{k+1}}-z_{i_k}|\leq L\eps\; \text{for all}\; 0\leq k\leq \ell \;(i_0=j_1;\;i_{\ell+1}=j_2) \Big\}\;.
\end{split}
\end{equation*}
In particular, points in $\cC_m$ are at most $m L \eps$ away from each other. Then we have
\begin{equation}
\label{e:no_singleton_volume}
    \big| \sS_{2n}^{c} \cap \{ |\vec{z}| \leq 2 \lambda \} \big| \lesssim_n (\eps \wedge \lambda)^{n|\fs|} \cdot \lambda^{n|\fs|}\;. 
\end{equation}
Furthermore, for every $\vec{z} \in \sS^c_{2n}$, there exists a partition $\pP$ of $\{1,\dots,2n\}$ such that each group $\set \in \pP$ contains at least two elements (hence $|\pP| \leq n$) and such that
\begin{equation}
\label{e:partition_cluster}
    \vec{z}|_{\set} \in \cC_{|\set|}
\end{equation}
for every $\set\in\pP$, where $\vec{z}|_{\set}$ denote the components in $\vec{z}$ that correspond to $\set \subset \{1, \dots, 2n\}$.

To prove Theorem~\ref{thm:main_result}, we first recall the results from \cite{HX19}. \cite[Remark~6.21]{HX19} told us that if we naively chaos expand the trigonometric function of the Gaussians on the right hand side of \eqref{e:main_expression} and control the high moments of each term, then we will end up with a similar bound to \eqref{eq:main_result} with a factor $e^{c (1+|\Btheta|)^2}$ for some $c>0$, where $|\Btheta| = |\theta_\fx| + |\theta_\fy|$.

To remove this inverse Gaussian factor, in \cite{HX19}, the authors clustered the points $\{x_i\}_{i=1}^{2n},\{y_i\}_{i=1}^{2n}$ into some clusters with radius $(1+|\Btheta|)^2 \eps$. Then we can take the points in the same cluster as the same point since the correlation of the Gaussian random variables $X_i,X_j$ in the same cluster can be bounded by $(1+|\Btheta|)^{-2}$ from below. We perform a chaos expansion over the clusters, which has less terms than direct chaos expansion. The choice of the clustering distance yields a similar bound to \eqref{eq:main_result} with a polynomial factor with respect to $\Btheta$.

In the case of one frequency, \cite{Xu18} obtained a frequency-independent bound by making the clustering distance to be $L \eps$ for sufficiently large $L$ independent of the frequency. In this case, one obtains a Gaussian factor $e^{-C \theta^2}$ for a sufficiently large $C$ if $\vec{x} \in \sS_{2n}$. This allows to cancel out the inverse Gaussian factor $e^{c\theta^2}$ in the case of one frequency. The argument also applies to two-frequency situation if the two frequencies are close to each other, that is, $|\theta_\fx| \sim |\theta_\fy|$ in our situation.

The obstacle comes when $|\theta_\fx| \gg |\theta_\fy|$. In this case, the similar argument to \cite{Xu18} yields a Gaussian factor $e^{-C \theta_\fx^2}$ if the point configuration is $\vec{x} \in \sS_{2n}$, which is sufficient to cancel out the growth $e^{c (|\theta_\fx|^2+|\theta_\fy|^2)}$. But if $\vec{x} \in \sS_{2n}^{c}$ and $\vec{y} \in \sS_{2n}$, then one only gets the decay factor $e^{-C |\theta_\fy|^2}$, which is insufficient to cancel out the growth $e^{c (|\theta_\fx|^2+|\theta_\fy|^2)}$.

In this situation, one observes that there is a small volume factor from the integration over $\vec{x} \in \sS_{2n}^{c}$ (see \eqref{e:no_singleton_volume}). Now, the natural way to bound the right hand side of \eqref{e:main_expression} independent of $\Btheta$ is to replace the expectation part by the obvious upper bound $1$. But if we do this, it turns out that the small volume factor $|\sS_{2n}^{c}|$ is still insufficient to match the correct power of $\eps$ in \eqref{eq:main_result}. Also, it does not seem obvious to obtain a frequency-independent bound for the expectation part except $1$. 

The main idea to overcome this issue is the localisation argument in Proposition~\ref{pr:main_idea} which, on the one hand keeps the small volume factor from $\vec{x} \in \sS_{2n}^{c}$, and on the other hand gains certain positive powers of $\eps$ from the factor $Y^{\diamond m_2} = \eps^{\frac{\alpha m_2}{2}} \Psi_\eps^{\diamond m_2}(y)$. 
Combining together the smallness from two different sources resolves the case with point configuration $\vec{x} \in \sS_{2n}^{c}$ and $\vec{y} \in \sS_{2n}$. Exactly the same argument applies to the case $|\theta_\fy| \gg |\theta_\fx|$ and $\vec{y} \in \sS_{2n}^{c}$. 

Now we turn to the proof of \eqref{eq:main_result}. The following property of the kernel will be used throughout the section. 

\begin{prop} \label{pr:kernel_bound}
    If $r_e=0$, the kernel $K$ satisfies the bound
    \begin{equation*}
        |K(x,y)| \lesssim \frac{1}{|x-y|^{|\fs|-\gamma}}\;.
    \end{equation*}
    If $r_e \geq 1$, then the kernel $K$ satisfies the bounds
    \begin{equation} \label{e:kernel_bound}
        |K(x,y)| \lesssim \left\{
        \begin{array}{ll}
        \frac{|x|^{r_e}}{|y|^{|\fs|-\gamma+ r_e}}\;, & |y| > 2|x|\;,\\
        \\
        \frac{1}{|x-y|^{|\fs|-\gamma}}\;, & \frac{|x|}{2} < |y| \leq 2 |x|\;,\\
        \\
        \frac{|x|^{r_e -1}}{|y|^{|\fs|-\gamma + r_e -1}}\;, & |y| \leq \frac{|x|}{2}
        \end{array} \right.
    \end{equation}
    uniformly over all $(x,y)$ within the above domain. 
\end{prop}
\begin{proof}
    Choose a smooth curve $\ell:[0,1]\mapsto\R^d$ connecting $-y$ and $x-y$ such that $\ell(0)=-y$, $\ell(1)=x-y$, it has length $L\lesssim |x|$, and $|\ell(s)|\in[|y|\wedge|x-y|,|y|\vee|x-y|]$ for every $s\in[0,1]$. The expression of $K$ implies that
    \begin{equation*}
        K(x,y) = \int\cdots\int_{0<s_{r_e}<\dots<s_1<1} D^{r_e}K_0\big(\ell(s_{r_e})\big) d\ell(s_{r_e})\cdots d\ell(s_1).   
    \end{equation*}
    Therefore, the desired bounds follow from \eqref{e:kernel_assumption}.
\end{proof}

\subsection{Some pointwise correlation bounds} \label{sec:correlation_bounds}

We want to improve the bound in \cite[Theorem~6.4]{HX19} to be independent of frequencies. \cite{Xu18} deals with the situation where there is only one frequency, while this article generalises to the case of two frequencies $\Btheta = (\theta_{\fx}, \theta_{\fy})$. The following three lemmas on pointwise correlation bounds will be used later on. We leave their proofs in Appendix~\ref{sec:singleton_proof}.

\begin{lem} 	\label{le:comparable_bound}
    We have the bound
    \begin{equ}	\label{e:comparable_bound}
		\bigg|\E	\prod_{i=1}^{2n} \partial_{\Btheta}^{\boldsymbol{r}}\fF(\Btheta,x_i,y_i)\bigg| \lesssim_n \E \prod_{i=1}^{2n} \bigg[ \Big( \sum_{k=m_1}^{ (m_1\vee m_2)+1} X_i^{\diamond k} \Big) \Big( \sum_{k=m_2}^{(m_1\vee m_2)+1}Y_i^{\diamond k} \Big) \bigg], 
	\end{equ}
	uniformly over all $\theta_{\fx}, \theta_{\fy} \in \R$ with $\frac{1}{100n(1+\Lambda^2)} |\theta_{\fy}| \leq |\theta_{\fx}| \leq 100n(1+\Lambda^2) |\theta_{\fy}|$, and all point configurations $\vec{x}$ and $\vec{y}$. 
\end{lem}
\begin{rmk}
    Note that Assumption~\ref{as:cor} implies that $\Psi_{\eps}$ is positively correlated. Therefore, the right hand side is positive by Wick's formula.
\end{rmk}

With this lemma, we can proceed as \cite[Section~3]{Xu18} to get Theorem~\ref{thm:main_result} in the case $\theta_{\fx}\sim\theta_{\fy}$. In the remainder of this section, we mainly focus on the case of $|\frac{\theta_{\fx}}{\theta_{\fy}}|$ being very large or very small.

\begin{lem} \label{le:singleton_bound}
	The bound \eqref{e:comparable_bound} holds uniformly over all $\theta_{\fx},\theta_{\fy}\in\R$ with $|\theta_{\fx}| > 100n(1+\Lambda^2) |\theta_{\fy}|$, and all point configurations $\vec{x}$ and $\vec{y}$ with the constraint $\vec{x} \in \sS_{2n}$.
	
	The same bound also holds uniformly over the range $|\theta_{\fy}| > 100 n (1 +\Lambda^2) |\theta_{\fx}|$ and the point configurations $\vec{x}$, $\vec{y}$ with the constraint $\vec{y} \in \sS_{2n}$. 
\end{lem}


If $|\theta_\fx| \gg |\theta_\fy|$, the case for the point configuration such that $\vec{x} \in \sS_{2n}^{c}$ and $\vec{y} \in \sS_{2n}$ is not covered by the above two lemmas (the configuration where both $\vec{x}$ and $\vec{y}$ are in $\sS_{2n}^c$ can be covered by other methods). In this situation, it turns out that we only need to deal with the extreme case where $x_1 = \cdots = x_{2n} = x$. The same is true for the point configuration $\vec{y} \in \sS_{2n}^c$ if $|\theta_\fy| \gg |\theta_\fx|$. These are covered by the following lemma.

\begin{lem} \label{le:fixed_bound}
If $|\theta_{\fx}| > 100n(1+\Lambda^2) |\theta_{\fy}|$, then we have the bound
	\begin{equation} \label{e:fixed_bound_alpha}
		\bigg|\E \prod_{i=1}^{2n} \partial_{\Btheta}^{\boldsymbol{r}}\fF(\Btheta,x,y_i)\bigg| \lesssim_n \eps^{- \alpha ((m_1\vee m_2)+1 )} \E   \prod_{i=1}^{2n} \Big( \sum_{k=m_2}^{(m_1\vee m_2)+1}Y_i^{\diamond k}\Big).
	\end{equation}
The proportionality constant is uniform over $\theta_{\fx}, \theta_{\fy}$ with $|\theta_{\fx}| > 100n(1+\Lambda^2) |\theta_{\fy}|$, and is also uniform over all point configurations $x$ and $\vec{y} = (y_1, \dots, y_{2n})$. 

Similarly, by swapping the roles of $x$ and $y$, the same bound is true uniformly over the parameters $|\theta_{\fy}| > 100n(1+\Lambda^2) |\theta_{\fx}|$ and all point configurations $\vec{x} = (x_1, \dots, x_{2n})$ and $y$. 
\end{lem}

\subsection{Proof of Theorem~\ref{thm:main_result}}

The main idea of the proof lies in the following proposition. The key part in the statement below is that in the integration over the small domain $\sS_{2n}^{c}$, one still gains extra powers of $\eps$ from $Y^{\diamond m_2}$ (and $X^{\diamond m_2}$), as explained at the beginning of Section~\ref{sec:proof}. 


\begin{prop} \label{pr:main_idea}
    Let $A: \R^d \times \R^d \rightarrow \R$ be absolutely integrable and supported in $\{(x,y):|x|,|y|\leq 2\}$. If $|\theta_{\fx}| > 100 n (1 + \Lambda^2) |\theta_{\fy}|$, then
    \begin{equation} \label{e:main_idea_1}
    \begin{split}
       \bigg\| \iint A(x,y) \d_{\Btheta}^{\r} &\fF(\Btheta,x,y) dxdy \bigg\|_{2n} \lesssim \bigg\| \sum_{k_1, k_2}  \iint |A(x,y)| X^{\diamond k_1}  Y^{\diamond k_2} dxdy \bigg\|_2\\
       &+ \eps^{-\frac{\alpha(m_1+m_2)}{2n}} \bigg( \int_{\sS_{2n}^{c}}\prod_{i=1}^{2n} \Big\| \int |A(x_i,y)| \cdot  Y^{\diamond m_2} dy \Big\|_{2} d\vec{x} \bigg)^{\frac{1}{2n}}.
    \end{split}
    \end{equation}
    Similarly, if $|\theta_{\fy}| > 100n(1+\Lambda^2) |\theta_{\fx}|$, then
    \begin{equation} \label{e:main_idea_2}
    \begin{split}
       \bigg\| \iint A(x,y) \d_{\Btheta}^{\r} &\fF(\Btheta,x,y) dxdy \bigg\|_{2n} \lesssim \bigg\| \sum_{k_1, k_2}  \iint |A(x,y)| X^{\diamond k_1} Y^{\diamond k_2} dxdy \bigg\|_2\\
       &+ \eps^{-\frac{\alpha(m_1+m_2)}{2n}} \bigg( \int_{\sS_{2n}^{c}} \prod_{i=1}^{2n} \bigg\| \int |A(x,y_i)| \cdot X^{\diamond m_2} dx \bigg\|_{2} d\vec{y} \bigg)^{\frac{1}{2n}}.
    \end{split}
    \end{equation}
    Both sums are over $m_1 \leq k_1 \leq (m_1 \vee m_2) + 1$ and $m_2 \leq k_2 \leq (m_1 \vee m_2) + 1$. The proportionality constants depend on $n$ and $\Lambda$ only.
\end{prop}

\begin{rmk} 
\label{rmk:Kernel_Asymmetric}
    We will later use this proposition with $A(x,y)$ being one of the following:
    \begin{equation*}
        K(x,y) \varphi^{\lambda}(x)\;, \qquad  K(x,y) \varphi^{\lambda}(x) \boldsymbol{1}_{|y| > 2 \lambda}\;, \qquad  K(x,y) \varphi^{\lambda}(x) \boldsymbol{1}_{|y| \leq 2 \lambda}\;,
    \end{equation*}
    But no specific assumptions on $A$ are made in the statement of the proposition, so the two bounds \eqref{e:main_idea_1} and \eqref{e:main_idea_2} are identical by swapping the roles of $x$ and $y$.  But we still state both of them for the following reason. Proposition~\ref{pr:main_idea} is an intermediate step to proving \eqref{eq:main_result}, and hence later we need to control the terms on the right hand sides of \eqref{e:main_idea_1} and \eqref{e:main_idea_2} by that of \eqref{eq:main_result}. It turns out that controlling their second terms (by the right hand side of \eqref{eq:main_result}) will need different arguments due to the kernel $K$ not being symmetric in $x$ and $y$. One can compare Propositions~\ref{pr:nelsontype_bound_x} and~\ref{pr:nelsontype_bound_y} below. In particular, we will consider $\{|y_i| \leq 2 \lambda\}$ and $\{|y_i| > 2\lambda\}$ separately when controlling the second term on the right hand side of \eqref{e:main_idea_2}, while we do not distinguish the location of $x_i$ for the second term in \eqref{e:main_idea_1}. 
    
\end{rmk}

\begin{proof}[Proof of Proposition~\ref{pr:main_idea}]
    The two bounds are identical up to a change of notation, so we prove \eqref{e:main_idea_1} only. 

    Divide the domain $\{|x| \leq 2\}$ into disjoint cubes of side length $L \eps$, and the total number of sub-cubes is bounded by a constant multiple of $\eps^{-|\fs|}$. We can further partition these cubes into at most $2^{d}$ groups denoted by $\Gamma_1, \Gamma_2, \dots$ such that for each $\Gamma_j$, the distance of any two different cubes in $\Gamma_j$ is at least $L \eps$. We have
    \begin{equation*}
        \bigg\| \iint A(x,y) \d_{\Btheta}^{\r} \fF(\Btheta,x,y) dxdy \bigg\|_{2n} \leq \sum_{j} \bigg\| \iint\limits_{x \in \Gamma_j} A(x,y) \d_{\Btheta}^{\r} \fF(\Btheta,x,y) dxdy \bigg\|_{2n}\;,
    \end{equation*}
    where the sum is taken over at most $2^{d}$ terms. Hence, it suffices to prove the bound \eqref{e:main_idea_1} for each $j$ on the right hand side above. We therefore fix any group of cubes which we denote by $\Gamma$ (with an abuse of notation), and use $Q_1, \dots, Q_N$ to denote the cubes in $\Gamma$ with $N \lesssim \eps^{-|\fs|}$. For any cube $Q$, write
    \begin{equation*}
        Z_Q := \iint\limits_{x \in Q} A(x,y) \d_{\Btheta}^{\r} \fF(\Btheta,x,y) dxdy\;.
    \end{equation*}
    Then we have
    \begin{equation*}
        \bigg\| \iint\limits_{x\in \Gamma} A(x,y) \d_{\Btheta}^{\r} \fF(\Btheta,x,y) dxdy \bigg\|_{2n}^{2n} = \sum_{\sigma} \E \big( Z_{Q_{\sigma(1)}} \cdots Z_{Q_{\sigma(2n)}} \big)\;,
    \end{equation*}
    where the sum is taken over all maps $\sigma: \{1, \dots, 2n\} \rightarrow \{1, \dots, N\}$. For any such map $\sigma$, we use the shorthand notation
    \begin{equation*}
        \qQ_{\sigma} := Q_{\sigma(1)} \times \cdots Q_{\sigma(2n)}\;. 
    \end{equation*}
    We now split the sum of $\sigma$ into two disjoint parts $\sigma \in \mathfrak{S}_1 \cup \mathfrak{S}_2$, where
    \begin{equation*}
        \mathfrak{S}_1:= \Big\{ \sigma: \exists k  \text{ such that} \; \sigma(k') \not = \sigma(k) \; \text{for every}\; k' \neq k \Big\}\;, 
    \end{equation*}
    and
    \begin{equation*}
        \mathfrak{S}_2:= \Big\{ \sigma: \text{for every} \;k,\;\exists k'\neq k \text{ such that } \sigma(k') = \sigma(k) \Big\}\;.
    \end{equation*}
    Note that since any two cubes in $\Gamma$ are at least $L \eps$ away from each other, $\sigma \in \mathfrak{S}_1$ implies $\qQ_{\sigma} \subset \sS_{2n}$, and $\sigma \in \mathfrak{S}_2$ implies $\qQ_{\sigma} \subset \sS_{2n}^{c}$. For the sum over $\mathfrak{S}_1$, we have
    \begin{equation*}
        \sum_{\sigma \in \mathfrak{S}_1} \E \prod_{k=1}^{2n} Z_{Q_{\sigma(k)}} =  \sum_{\sigma \in \mathfrak{S}_1} \; \idotsint\limits_{\vec{x} \in \qQ_{\sigma}} \prod_{i=1}^{2n} A(x_i, y_i) \cdot \E \bigg( \prod_{i=1}^{2n} \d_{\Btheta}^{\r} \fF(\Btheta, x_i, y_i) \bigg) d \vec{y} d \vec{x}\;.
    \end{equation*}
    Since $\qQ_{\sigma} \subset \sS_{2n}$ for $\sigma \in \mathfrak{S}_1$, by Lemma~\ref{le:singleton_bound}, we have
    \begin{equation*}
    \begin{split}
        \Big| \sum_{\sigma \in \mathfrak{S}_1} \E \prod_{k=1}^{2n} Z_{Q_{\sigma(k)}} \Big| &\lesssim \idotsint \prod_{i=1}^{2n} |A(x_i, y_i)| \; \E \prod_{i=1}^{2n} \Big(\sum_{k_1, k_2} X_i^{\diamond k_1} Y_i^{\diamond k_2}\Big) d \vec{x} d \vec{y}\\
        &= \bigg\| \iint |A(x,y)| \cdot \sum_{k_1, k_2} X^{\diamond k_1} Y^{\diamond k_2} dx dy \bigg\|_{2n}^{2n}\;,
    \end{split}
    \end{equation*}
    where in the first inequality we have enlarged the domain of integration to all $\vec{x}$ and $\vec{y}$ so that the term on the right hand side is exactly the $2n$-th moment of an integral. The sum is taken over $k_1$, $k_2$ in the range of the statement. Note that the integral
    \begin{equation*}
        \mathfrak{A}:= \iint |A(x,y)| \cdot \sum_{k_1, k_2} X^{\diamond k_1} Y^{\diamond k_2} dx dy
    \end{equation*}
    lives in the first $\big(2(m_1\vee m_2)+2\big)$-th Wiener chaos space. Then by hypercontractivity estimates (see \cite[Section~1.4.3]{Nua06}), we have
    \begin{equation*}
    \begin{split}
        \| \mathfrak{A} \|_{2n}^{2n} &\leq \sum_{\ell=0}^{2(m_1\vee m_2)+2} \| J_\ell \mathfrak{A} \|_{2n}^{2n} \lesssim_{m_1,m_2} \sum_{\ell=0}^{2(m_1\vee m_2)+2} \| J_\ell \mathfrak{A} \|_{2}^{2n} \lesssim_{m_1,m_2} \| \mathfrak{A} \|_{2}^{2n},
    \end{split}
    \end{equation*}
    where $J_n$ is the projection on the $n$-th Wiener chaos space and the last inequality follows from the $L^2$-orthogonality of the Wiener chaos. This gives the first term on the right hand side of \eqref{e:main_idea_1}.
    
    Now we turn to the sum in $\mathfrak{S}_2$. By H\"older and Minkowski inequalities and that $\qQ_{\sigma} \subset \sS_{2n}^{c}$ for $\sigma \in \mathfrak{S}_2$, we get
    \begin{equation*}
        \sum_{\sigma \in \mathfrak{S}_2} \E \prod_{k=1}^{2n} Z_{Q_{\sigma(k)}} \leq \sum_{\sigma \in \mathfrak{S}_2} \prod_{k=1}^{2n} \|Z_{Q_{\sigma(k)}}\|_{2n} \leq \int_{\sS_{2n}^{c}} \prod_{i=1}^{2n} \bigg\| \int A(x_i, y) \d_{\Btheta}^{\r} \fF(\Btheta, x_i, y) dy \bigg\|_{2n} d \vec{x}\;.
    \end{equation*}
    By Lemma~\ref{le:fixed_bound} and hypercontractivity estimates, we have
    \begin{equation*}
    \begin{split}
       \bigg\| \int|A(x_i,y)| \d_{\Btheta}^{\r} \fF(\Btheta,x_i,y) dy \bigg\|_{2n} &= \bigg( \int \prod_{k=1}^{2n} |A(x_i,y_k)| \E \prod_{k=1}^{2n} \d_{\Btheta}^{\r} \fF(\Btheta,x_i,y_k) d\vec{y} \bigg)^{\frac{1}{2n}}\\
       &\lesssim \eps^{-\frac{\alpha(m_1+m_2)}{2n}} \bigg\| \int |A(x_i,y)| \sum_{k=m_2}^{(m_1\vee m_2)+1} Y^{\diamond k}dy \bigg\|_{2}.
    \end{split}
    \end{equation*}
    By Wick's formula and Assumption~\ref{as:cor}, for $m_2 \leq k \leq (m_1 \vee m_2) +1$, we have
    \begin{equation*}
        \E \big( Y_1^{\diamond k} Y_{2}^{\diamond k} \big) = (\E Y_1Y_2)^k \lesssim (\E Y_1Y_2)^{m_2} = \E \big( Y_1^{\diamond m_2} Y_{2}^{\diamond m_2} \big)\;.
    \end{equation*}
    We can thus replace the sum of $Y^{\diamond k}$ over $m_2 \leq k \leq (m_1\vee m_2) +1$ by the single term $Y^{\diamond m_2}$. This completes the proof of the proposition. 
\end{proof}

The following result follows from \cite[Theorem~A.3]{HQ}, which gives the desired bound of the first terms on the right hand side of \eqref{e:main_idea_1} and \eqref{e:main_idea_2} with $A(x,y)=\varphi^{\lambda}(x) K(x,y)$.

\begin{prop}\label{pr:poly_bound}
    For $K, m_1, m_2$ satisfying our assumptions and any $\eta > 0$, we have
    \begin{equ}\label{e:poly_bound}
        \bigg\| \iint |\varphi^{\lambda}(x) K(x,y)| \Big( \sum_{k=m_1}^{(m_1\vee m_2)+1} X^{\diamond k} \Big) \Big( \sum_{k=m_2}^{(m_1\vee m_2)+1} Y^{\diamond k}\Big) dxdy \bigg\|_2 \lesssim \Big( \frac{\eps}{\lambda} \Big)^{\frac{\alpha(m_1+m_2)}{2}} \lambda^{\gamma-\eta},
    \end{equ}
    where the proportionality constant is independent of $\lambda$ and $\eps$.
\end{prop}

The following two propositions are devoted to estimating the second terms on the right hand side of \eqref{e:main_idea_1} and \eqref{e:main_idea_2} with $A(x,y)$ being $\varphi^{\lambda}(x) K(x,y)$ and its variants. 

\begin{prop}
\label{pr:nelsontype_bound_x}
	Let $\gG(x) := \|\int |K(x,y)| Y^{\diamond m_2}dy\|_{2}$ and $\eta>0$ be sufficiently small. If $\gamma \leq \frac{\alpha m_2}{2}$, then
    \begin{equation*}
        \gG(x) \lesssim \eps^{\gamma-\eta}\;.
    \end{equation*}
    If $\gamma > \frac{\alpha m_2}{2}$, then we have
    \begin{equation*}
        \gG(x) \lesssim \eps^{\frac{\alpha m_2}{2}} |x|^{\gamma - \frac{\alpha m_2}{2} - \eta}\;.
    \end{equation*}
	The proportionality constants are independent of $\eps$ and $\lambda$ in both situations. 
\end{prop}

\begin{proof}
    If $\gamma \leq \frac{\alpha m_2}{2}$, then $r_e = 0$, and we have
    \begin{equation*}
        |\gG(x)|^2 \lesssim \int \int \frac{1}{|y_1 - x|^{|\fs|-\gamma} |y_2 - x|^{|\fs|-\gamma}} \cdot \frac{\eps^{\alpha m_2}}{(|y_1 - y_2| + \eps)^{\alpha m_2}} d y_1 d y_2 
    \end{equation*}
    Since $\alpha m_2 > 2\gamma-\eta$, we have
    \begin{equation*}
        |\gG(x)|^2 \lesssim  \int \frac{\eps^{2\gamma-\eta}}{|y_1 - x|^{|\fs|-\gamma}} \bigg( \int \frac{1}{|y_2 - x|^{|\fs|-\gamma} |y_1 - y_2|^{2\gamma-\eta}} d y_2 \bigg) d y_1 \lesssim \eps^{2\gamma - \eta}\;.
    \end{equation*}
    If $\gamma > \frac{\alpha m_2}{2}$, then
    \begin{equation*}
        \gamma - \frac{\alpha m_2}{2} \leq r_e < \gamma - \frac{\alpha m_2}{2} + 1\;.
    \end{equation*}
    We have
    \begin{equation*}
        \gG(x) \leq \gG_1(x) + \gG_2(x) + \gG_3(x)\;,
    \end{equation*}
    which correspond to domains of integration $\{|y| > 2 |x|\}$, $\{\frac{|x|}{2} \leq |y| < 2 |x|\}$ and $\{|y| \leq \frac{|x|}{2}\}$ respectively. 

    By the first bound in \eqref{e:kernel_bound}, we have
    \begin{equation*}
        |\gG_1(x)|^{2} \lesssim \eps^{\alpha m_2} |x|^{2 r_e} \int_{|y_1| > 2|x|} \frac{1}{|y_1|^{|\fs|-\gamma + r_e}} \bigg( \int_{|y_2| > 2|x|} \frac{1}{|y_2|^{|\fs|-\gamma + r_e} |y_1 - y_2|^{\alpha m_2}} d y_2 \bigg) d y_1\;.
    \end{equation*}
    Since $r_e \geq \gamma - \frac{\alpha m_2}{2}$, we have $\alpha m_2 + r_e - \gamma + \eta > 0$. Hence, the integral in the parenthesis above is bounded by $\frac{1}{|y_1|^{\alpha m_2 + r_e + \eta - \gamma}}$ if $\gamma > r_e$, and bounded by $\frac{1}{|x|^{r_e - \gamma + \eta}} \cdot \frac{1}{|y_1|^{\alpha m_2 -\eta}}$ if $\gamma \leq r_e$. Both situations yield
    \begin{equation*}
        |\gG_1 (x)|^{2} \lesssim \eps^{\alpha m_2} |x|^{2\gamma - \alpha m_2 - \eta}\;.
    \end{equation*}
    As for $\gG_2$, by the second bound in \eqref{e:kernel_bound}, we have
     \begin{equation*}
        |\gG_2(x)|^{2} \lesssim \int_{\frac{|x|}{2} < |y_1| \leq 2|x|} \frac{\eps^{\alpha m_2}}{|y_1 - x|^{|\fs|-\gamma}} \bigg( \int_{\frac{|x|}{2} < |y_2| \leq 2|x|} \frac{1}{|y_2 - x|^{|\fs|-\gamma} |y_2 - y_1|^{\alpha m_2}} d y_2 \bigg) d y_1\;. 
    \end{equation*}
    The integral in the parenthesis above is bounded by $\frac{|x|^{\eta}}{|y_1 - x|^{\alpha m_2 - \gamma + \eta}}$ if $\alpha m_2 \geq \gamma$, and bounded by $|x|^{\gamma - \alpha m_2}$ if $\alpha m_2 < \gamma$. In both situations, we have
    \begin{equation*}
        |\gG_2 (x)|^{2} \lesssim \eps^{\alpha m_2} |x|^{2 \gamma - \alpha m_2 - \eta}\;.
    \end{equation*}
    Finally for $\gG_3$, using the third bound in \eqref{e:kernel_bound}, we have
    \begin{equation*}
        |\gG_3(x)|^{2} \lesssim \int_{|y_1| \leq \frac{|x|}{2}} \frac{\eps^{\alpha m_2} |x|^{2(r_e - 1)}}{|y_1|^{|\fs|-\gamma + r_e - 1}} \bigg( \int_{|y_2| \leq \frac{|x|}{2}} \frac{1}{|y_2|^{|\fs|-\gamma + r_e -1} |y_2 - y_1|^{\alpha m_2}} d y_2 \bigg) d y_1\;.
    \end{equation*}
    Note that by assumption, $\alpha m_2 - \gamma + r_e - 1$ is always     less than $|\fs|$. If $\alpha m_2 - \gamma + r_e - 1 \geq 0$, then the integral in the parenthesis above is bounded by $\frac{1}{|y_1|^{\alpha m_2 - \gamma + r_e - 1 + \eta}}$; otherwise it is bounded by $|x|^{\gamma - \alpha m_2 - (r_e - 1)}$. In both situations, we have
    \begin{equation*}
        |\gG_3(x)|^{2} \lesssim \eps^{\alpha m_2} |x|^{2 \gamma - \alpha m_2 - \eta}. 
    \end{equation*}
    This concludes the proof of the proposition. 
\end{proof}

\begin{rmk}
    The appearance of the small power $\eta>0$ comes from the logarithm bounds of some integrals. In the above proof, some of the appearances of the small power $\eta$ are not necessary. But note that the bounds with extra power $\eta>0$ still hold since $1\lesssim \frac{1}{|y|^\eta}$, where $y$ is restricted to $|y|\leq2$. We do not distinguish different cases to avoid checking which integral requires logarithm bound. The same convention is used in the subsequent proof.
\end{rmk}

\begin{prop}\label{pr:nelsontype_bound_y}
    Let $\hH(y) := \|\int |K(x,y)| \cdot |\varphi^{\lambda}(x)| \cdot X^{\diamond m_1} dx\|_{2}$. We have the bounds
    \begin{equ} \label{e:nelsontype_bound_y1}
       \hH(y) \lesssim \frac{\eps^{\frac{\alpha m_1}{2}}\lambda^{r_e -\frac{\alpha m_1}{2}}}{|y|^{|\fs|-\gamma+r_e}}\;, \qquad |y| > 2 \lambda\;,
    \end{equ} 
    and
    \begin{equation} \label{e:nelsontype_bound_y2}
       \hH(y) \lesssim \boldsymbol{1}_{\{r_e \geq 1\}} \cdot \frac{\eps^{\frac{\alpha m_1}{2}} \lambda^{r_e-\frac{\alpha m_1}{2}-1}}{|y|^{|\fs|-\gamma+r_e-1}} + \eps^{(\gamma \wedge \frac{\alpha m_1}{2}) - \eta} \lambda^{\gamma-|\fs|- (\gamma \wedge \frac{\alpha m_1}{2}) - \eta}\;, \quad |y| \leq 2\lambda
    \end{equation}
    uniformly over $\eps, \lambda \in (0,1)$ and over $y$ in the above domain. 
\end{prop}
\begin{proof}
    The arguments are very similar to those in proving Proposition~\ref{pr:nelsontype_bound_x}, so we omit the technical details here.
\end{proof}

\begin{rmk}
    The appearances of $\lambda$ on the right hand side of \eqref{e:nelsontype_bound_y1} and \eqref{e:nelsontype_bound_y2} are due to the presence of $\varphi^{\lambda}$ in the definition of $\hH(y)$, which restricts the integration domain to a box of size $\lambda$. In the previous proposition the integration domain is of constant size, so there is no $\lambda$ appearing in the bounds.
\end{rmk}

\begin{rmk}
    One can also use the bounds
    \begin{equation*}
        \frac{\eps^{\alpha m_2}}{(|y_1 - y_2| + \eps)^{\alpha m_2}}  \lesssim 1 \qquad \text{and} \qquad \frac{\eps^{\alpha m_1}}{(|x_1 - x_2| + \eps)^{\alpha m_1}} \lesssim 1
    \end{equation*}
    to improve both Propositions~\ref{pr:nelsontype_bound_x} and~\ref{pr:nelsontype_bound_y} by replacing $\eps$ with $\eps \wedge \lambda$ in the statements. But the current statements are already sufficient for the proof of the main theorem. 
\end{rmk}

\begin{lem} \label{le:region_smallness_int1}
    For arbitrarily small $\eta > 0$, we have the bound
    \begin{equation} \label{e:region_smallness_int1}
		\int_{\sS_{2n}^{c}} \prod_{i=1}^{2n} \frac{\boldsymbol{1}_{|y_i|\geq 2\lambda}}{|y_i|^{|\fs|-\gamma+r_e}} d\vec{y} \lesssim \lambda^{2n(\gamma-r_e-\eta)} \bigg( \frac{\eps}{\lambda} \bigg)^{n \alpha m_2}, 
	\end{equation}
    where the proportionality constant depends on $\eta$. 
\end{lem}
\begin{proof}
    By \eqref{e:partition_cluster}, we have
    \begin{equation*}
        \int_{\sS_{2n}^{c}} \prod_{i=1}^{2n} \frac{\boldsymbol{1}_{|y_i|\geq 2\lambda}}{|y_i|^{|\fs|-\gamma+r_e}} d\vec{y} \leq \sum_{\pP} \prod_{\set \in \pP} \int_{\cC_{|\set|}} \prod_{i \in \set} \frac{1}{|y_i|^{|\fs|-\gamma+r_e}} \boldsymbol{1}_{|y_i| \geq 2 \lambda} d \vec{y}_{\set}\;,
    \end{equation*}
    where $d \vec{y}_{\set} = \prod_{i \in \set} d y_i$, and each $\set$ in $\pP$ satisfies $|\set| \geq 2$. 
    
    If $\lambda \geq 4nL\eps$, then we have
    \begin{equation*}
        \begin{split}
        \int_{\cC_{|\set|}} \prod_{i \in \set} \frac{\boldsymbol{1}_{|y_i| \geq 2 \lambda}}{|y_i|^{|\fs|-\gamma+r_e}}  d \vec{y}_{\set} &\lesssim \eps^{(|\set|-1)|\fs|} \int_{|y| > \lambda} \frac{1}{|y|^{|\set| (|\fs|-\gamma+r_e)}} dy\\
        &\lesssim \big(\eps^{|\fs|} \lambda^{\gamma - r_e - |\fs| - \eta} \big)^{|\set|} \; (\lambda / \eps)^{|\fs|}\;,
        \end{split}
    \end{equation*}
    where the first bound above follows from the definition of $\cC_{\set}$ and the second bound holds since $2(|\fs|-\gamma+r_e) \geq |\fs|$ and $|\set| \geq 2$. 

    Since $\sum_{\set \in \pP} |\set| = 2n$ and $|\pP| \leq n$, multiplying the above bound over $\set \in \pP$ and summing over all partitions $\pP$ with group size at least two gives
    \begin{equation*}
        \int_{\sS_{2n}^{c}} \prod_{i=1}^{2n} \frac{\boldsymbol{1}_{|y_i|\geq 2\lambda}}{|y_i|^{|\fs|-\gamma+r_e}} d\vec{y} \lesssim \eps^{n|\fs|} \lambda^{2n(\gamma - r_e - \frac{|\fs|}{2} - \eta)}\;.
    \end{equation*}
    Now we turn to the case $\lambda < 4 n L \eps$. Take arbitrary $\set$ with $|\set| \geq 2$. We assume without loss of generality that $1 \in \set$. For integration of $\vec{y}_{\set}$ over $\cC_{|\set|}$, we separate the two domains $\{|y_1| > 4 n L \eps\}$ and $\{|y_1| \leq 4nL\eps\}$. For the first domain, we have
    \begin{equation*}
        \int_{\cC_{|\set|} \cap \{|y_1| > 4nL\eps\}} \prod_{i \in \set} \frac{\boldsymbol{1}_{|y_i| \geq 2 \lambda}}{|y_i|^{|\fs|-\gamma+r_e}} d \vec{y}_{\set} \leq \eps^{(|\set|-1)|\fs|} \int_{|y| > nL\eps} \frac{1}{|y|^{|\set|(|\fs|-\gamma+r_e)}} d y \lesssim \eps^{|\set|(\gamma-r_e-\eta)}\;.
    \end{equation*}
    As for the second domain, we have
    \begin{equation*}
        \int_{\cC_{|\set|} \cap \{|y_1| \leq 4nL\eps\}} \prod_{i \in \set} \frac{\boldsymbol{1}_{|y_i| \geq 2 \lambda}}{|y_i|^{|\fs|-\gamma+r_e}} d \vec{y}_{\set} \lesssim \prod_{i \in \set} \int_{\lambda \leq |y| \leq 6nL\eps} \frac{1}{|y_i|^{|\fs|-\gamma+r_e}} d y_i\;,
    \end{equation*}
    which is bounded by $\eps^{|\set|(\gamma-r_e)}$ if $\gamma > r_e$, and bounded by $\lambda^{|\set|(\gamma-r_e-\eta)}$ if $\gamma \leq r_e$. Since $\lambda \lesssim \eps$, we see the left hand side of \eqref{e:region_smallness_int1} is bounded by $\eps^{2n(\gamma-r_e-\eta)}$ if $\gamma > r_e$, and by $\lambda^{2n(\gamma-r_e-\eta)}$ if $\gamma \leq r_e$. Combining these bounds with the previous one with $\lambda \geq 4nL\eps$, and using the relation $\gamma-r_e \leq \frac{\alpha m_2}{2} < \frac{|\fs|}{2}$, we conclude \eqref{e:region_smallness_int1}. 
\end{proof}

\begin{lem}
\label{le:region_smallness_int2}
    Suppose $r_e \geq 1$. Then for arbitrarily small $\eta>0$, we have
    \begin{equation*}
        \int_{\sS_{2n}^{c}} \prod_{i=1}^{2n} \frac{\boldsymbol{1}_{\{|y_i| \leq 2 \lambda \}}}{|y_i|^{|\fs|-\gamma + r_e - 1}}  d \vec{y} \; \lesssim (\eps \wedge \lambda)^{2n(\gamma-r_e+1-\eta)}\;,
	\end{equation*}
	where the proportionality constant is independent of $\eps$ and $\lambda$.
\end{lem}
\begin{proof}
    If $\lambda \leq 4nL \eps$, then the quantity is bounded by
    \begin{equation*}
        \prod_{i=1}^{2n} \int_{|y_i| \leq 2\lambda} \frac{1}{|y_i|^{|\fs|-\gamma+r_e - 1}} d y_i \; \lesssim \lambda^{2n(\gamma- r_e +1)}, 
    \end{equation*}
    where the second inequality holds since $\gamma - r_e + 1 > \frac{\alpha m_2}{2}$ and hence is positive. 

    Now we consider the case $\lambda > 4nL \eps$. Same as before, we have
    \begin{equation*}
        \int_{\sS_{2n}^{c}} \prod_{i=1}^{2n} \frac{\boldsymbol{1}_{|y_i| \leq 2\lambda}}{|y_i|^{|\fs|-\gamma+r_e-1}} d\vec{y} \leq \sum_{\pP} \prod_{\set \in \pP} \int_{\cC_{|\set|}} \prod_{i \in \set} \frac{\boldsymbol{1}_{|y_i| \leq 2 \lambda}}{|y_i|^{|\fs|-\gamma+r_e-1}}  d \vec{y}_{\set}\;.
    \end{equation*}
    Take arbitrary $\pP$ and $\set \in \pP$, and assume without loss of generality that $y_1 \in \set$. We decompose the domain of integration $\cC_{|\set|}$ as
    \begin{equation*}
        \cC_{|\set|} = \big( \cC_{|\set|} \cap \{|y_1| \leq 4nL\eps\} \big) \cup \big( \cC_{|\set|} \cap \{|y_1| > 4nL\eps\} \big)\;.
    \end{equation*}
    For the former, we have
    \begin{equation*}
        \int_{\cC_{|\set|} \cap \{|y_1| \leq 4nL\eps\}} \prod_{i \in \set} \frac{\boldsymbol{1}_{|y_i| \leq 2 \lambda}}{|y_i|^{|\fs|-\gamma+r_e-1}}  d \vec{y}_{\set} \lesssim \prod_{i \in \set} \int_{|y_i| \leq 6nL\eps} \frac{1}{|y_i|^{|\fs|-\gamma+r_e-1}} d y_i \lesssim \eps^{|\set|(\gamma-r_e+1)}\;.
    \end{equation*}
    For the latter, we have
    \begin{equation*}
        \int_{\cC_{|\set|} \cap \{|y_1|>4nL\eps\}} \prod_{i \in \set} \frac{\boldsymbol{1}_{|y_i| \leq 2 \lambda}}{|y_i|^{|\fs|-\gamma+r_e-1}}  d \vec{y}_{\set} \lesssim \eps^{(|\set|-1)|\fs|} \int_{|y|>nL\eps} \frac{1}{|y|^{|\set|(|\fs|-\gamma+r_e-1)}} dy \lesssim \eps^{|\set|(\gamma-r_e+1-\eta)}\;, 
    \end{equation*}
    where we have used the definition of $\cC_{|\set|}$ and that $2(|\fs|-\gamma+r_e-1) \geq |\fs|$. This gives
    \begin{equation*}
        \int_{\sS_{2n}^{c}} \prod_{i=1}^{2n} \frac{\boldsymbol{1}_{\{|y_i| \leq 2 \lambda \}}}{|y_i|^{|\fs|-\gamma + r_e - 1}}  d \vec{y} \lesssim \eps^{2n(\gamma-r_e+1-\eta)}
    \end{equation*}
    if $\lambda > 4nL\eps$. This concludes the proof. 
\end{proof}

Now we have all the ingredients to prove Theorem~\ref{thm:main_result}. 

\begin{proof} [Proof of Theorem~\ref{thm:main_result}]
If $|\theta_{\fx}| > 100 n (1 + \Lambda^2) |\theta_{\fy}|$, then we have
\begin{equation*}
    \begin{split}
    &\phantom{111}\bigg( \int_{\sS_{2n}^{c}} \prod_{i=1}^{2n} |\varphi^{\lambda}(x_i)| \prod_{i=1}^{2n} \Big\| \int |K(x_i, y)| Y^{\diamond m_2} dy \Big\|_{2} d \vec{x} \bigg)^{\frac{1}{2n}}\\
    &\lesssim \lambda^{-|\fs|} \cdot \eps^{\gamma \wedge \frac{\alpha m_2}{2} - \eta} \lambda^{\gamma - (\gamma \wedge \frac{\alpha m_2}{2} + \eta)} \cdot \big| \sS_{2n}^{c} \cap \{|\vec{x}| \leq \lambda \} \big|\\
    &\lesssim \eps^{\frac{\alpha(m_1 + m_2)}{2} - \eta} \lambda^{\gamma - \frac{\alpha(m_1 + m_2)}{2} - \eta}\;,
    \end{split}
\end{equation*}
where the first inequality follows from Proposition~\ref{pr:nelsontype_bound_x} and the second one follows from \eqref{e:no_singleton_volume} and a direct computation of the exponents and relative sizes of $\eps$ and $\lambda$. 

If $|\theta_{\fy}| > 100 n (1 + \Lambda^2) |\theta_{\fx}|$, then we need to bound
\begin{equation*}
    \bigg( \int_{\sS_{2n}^{c}} \prod_{i=1}^{2n} \Big( \hH(y_i) \boldsymbol{1}_{|y_i| > 2\lambda} \Big) d \vec{y}  \bigg)^{\frac{1}{2n}} \quad \text{and} \quad \bigg( \int_{\sS_{2n}^{c}} \prod_{i=1}^{2n} \Big( \hH(y_i) \boldsymbol{1}_{|y_i| \leq 2\lambda} \Big) d \vec{y} \bigg)^{\frac{1}{2n}}\;.
\end{equation*}
We can bound the first quantity by $\eps^{\frac{\alpha(m_1 + m_2)}{2}} \lambda^{\gamma - \frac{\alpha(m_1 + m_2)}{2} - \eta}$ by Proposition~\ref{pr:nelsontype_bound_y} and Lemma~\ref{le:region_smallness_int1}. The second quantity can be controlled in the same way by Proposition~\ref{pr:nelsontype_bound_y} and Lemma~\ref{le:region_smallness_int2}. The proof of Theorem~\ref{thm:main_result} is complete. 
\end{proof}

\section{Application to weak universality problems}
\label{sec:spde}

In this section, we prove Theorems~\ref{thm:kpz} and~\ref{thm:phi43} with $F \in \cC^{2,\beta}_{\pP}$ and $G \in \cC^{3,\beta}_{\pP}$ respectively with the growth power $M$ as in Definition~\ref{def:nonlinearity_space}. With the theory of regularity structures, both theorems follow from two ingredients: well-posedness and convergence of the abstract equation, which is the deterministic part, and convergence of certain stochastic objects, which is the probabilistic part. 

For the KPZ equation, it was shown in \cite[Theorem 3.7]{HX19} that the abstract equation for \eqref{e:kpz_macro} is well-posed if $F \in \cC^{4,\beta}_{\pP}$, but very slight modification of the arguments there will reduce the requirement to $F \in \cC^{2,\beta}_{\pP}$. Exactly the same argument could be used to show that the abstract equation for \eqref{e:phi43_macro} is well-posed if $G \in \cC^{3,\beta}_{\pP}$. We summarise in the following theorem. 

\begin{thm} \label{thm:abstract_eq}
    If $F \in \cC^{2,\beta}_{\pP}$, then the abstract equation for \eqref{e:kpz_macro} in regularity structures is well-posed. 
    Similarly, if $G \in \cC^{3,\beta}_{\pP}$, then the abstract equation for \eqref{e:phi43_macro} in regularity structures is well-posed. 
\end{thm}
\begin{proof}
    \cite[Theorem 3.7]{HX19} gives the well-posedness of the abstract equation for \eqref{e:kpz_macro} under $F \in \cC^{2,\beta}_{\pP}$. The only places where the third and fourth derivatives are used is an intermediate value theorem for the quantity
    \begin{equation*}
        R(x,y) := F(x+y) - F(x) - F'(x) y - \frac{1}{2} F''(x) y^2
    \end{equation*}
    and the difference $R(x,y) - R(x,z)$. One needs to show the bounds
    \begin{equation*}
    \begin{split}
        |R(x,y)| &\lesssim (1 + |x| + |y|)^{M} |y|^{2+\beta}\;,\\
        |R(x,y) - R(x,z)| &\lesssim (1 + |x| + |y| + |z|)^{M+1+\beta} |y-z|\;.
    \end{split}
    \end{equation*}
    Instead of using the intermediate value theorem with $F^{(3)}$, the first bound follows from
    \begin{equation*}
        R(x,y) = \iint\limits_{x<u<v<x+y} \big( F''(u) - F''(x) \big) du dv
    \end{equation*}
    and H\"older continuity of $F''$. Similarly, instead of using the intermediate value theorem with four derivatives of $F$, one has
    \begin{equation*}
        R(x,y) - R(x,z) = \int_{x+z}^{x+y} \bigg( \int_{x}^{v} \big( F''(u) - F''(x) \big) du \bigg) dv\;.
    \end{equation*}
    One then has the second bound with only H\"older continuity of $F''$. This gives the well-posedness of the abstract equation with $F \in \cC^{2,\beta}_{\pP}$ for the KPZ case. The same argument also gives the well-posedness of the abstract equation with $G \in \cC^{3,\beta}_{\pP}$ for the dynamical $\Phi^4_3$ case. 
\end{proof}

In the rest of the section, we will show the convergence of the corresponding stochastic objects in KPZ and $\Phi^4_3$ under $F \in \cC^{2,\beta}_{\pP}$ and $G \in \cC^{3,\beta}_{\pP}$ respectively. These convergences, together with Theorem~\ref{thm:abstract_eq}, complete the proofs of Theorem~\ref{thm:kpz} and~\ref{thm:phi43}.

\subsection{Preliminaries}

For $N \in \N$ and $\Btheta = (\theta_1, \dots, \theta_N) \in\R^{N}$, let $\fR_{\Btheta}$ be the cube with side length $2$ centred at $\Btheta$ and let $\fR_{\theta_i}$ be the interval with length $2$ centred at $\theta_i$. Following \cite[Section 4.3]{HX19}, for every integer $M$, open set $\Omega\subset\R^{N}$ and distribution $\Upsilon$ on $\R^{N}$, we define the norm
\begin{equation*}
    \| \Upsilon \|_{M,\Omega} := \sup_{\phi: \|\phi\|_{\bB_{M}(\Omega)} \leq 1} |\scal{\Upsilon,\phi}|,
\end{equation*}
where the norm $\|\cdot\|_{\bB_{M}(\Omega)}$ on $\cC_{c}^{\infty}(\Omega)$ is defined by
\begin{equation*}
    \|\phi\|_{\bB_{M}(\Omega)} := \sup_{\r : |\r|_{\infty} \leq M} \sup_{x\in\Omega} |\d^{\r}\phi(x)|.
\end{equation*}
The following lemmas from \cite[Section 4.3]{HX19} are needed in the rest of this section.

\begin{lem}
    \label{lem:local_decompose}
    For every distribution $\Upsilon$ on $\R^N$ and $\Phi\in\cC^{\infty}(\R^N)$, we have
    \begin{equation*}
        |\scal{\Upsilon,\Phi}| \lesssim_{M} \sum_{\K\in\Z^N} \| \Upsilon \|_{M+2,\fR_{\K}} \sup_{\r:|\r|_{\infty} \leq M+2} \sup_{\Btheta\in\fR_{\K}} |\d^{\r}_{\Btheta} \Phi(\Btheta)|\;.
    \end{equation*}
\end{lem}
\begin{proof}
    Same as \cite[Proposition~4.10]{HX19}.
\end{proof}

\begin{lem}
\label{lem:local_bound}
    Suppose $F\in\cC^{2,\beta}_{\pP}$. Let $\l = (\l_1, \dots, \l_d) \in\N^{d}$ and
    \begin{equation*}
        \Upsilon = \otimes_{i=1}^{d} \widehat{F^{(\ell_{i})}},\qquad \Upsilon_{\delta} = \otimes_{i=1}^{d} \widehat{F_{\delta}^{(\ell_{i})}}.
    \end{equation*}
    For $\K\in\Z^{d}$ we have the bound
    \begin{equation*}
        \| \Upsilon_{\delta} \|_{M+2,\fR_{\K}} \lesssim \prod_{i=1}^{d} (1+|K_i|)^{-2-\beta+\ell_i}\;,
    \end{equation*}
    where the proportionality constant is independent of $\K$. For the difference $\Upsilon - \Upsilon_{\delta}$, we have the bound
    \begin{equation*}
        \| \Upsilon - \Upsilon_{\delta} \|_{M+2,\fR_{\K}} \lesssim \delta^{\omega} \prod_{i=1}^{d} (1+|K_i|)^{-2-\beta+\ell_i+\omega}
    \end{equation*}
    uniformly over $\K\in\Z^d$, $\delta\in(0,1)$ and $\omega\in(0,1)$.
\end{lem}
\begin{proof}
    We only provide the proof of the first estimate, the second estimate is similar. By \cite[Lemma~4.5]{HX19}, we have
    \begin{equation*}
        \| \Upsilon_{\delta} \|_{M+2,\fR_{\K}} \lesssim \prod_{i=1}^{d} \|\widehat{F_\delta^{(\ell_i)}}\|_{M+2,\fR_{K_i}}\;.
    \end{equation*}
    Then the desired result follows from
    \begin{equation*}
        \|\widehat{F_\delta^{(\ell_i)}}\|_{M+2,\fR_{K_i}} \lesssim (1+|K_i|)^{-2-\beta+\ell_i},
    \end{equation*}
    which is a direct corollary of \cite[Lemma~4.8]{HX19}.
\end{proof}

\begin{lem}
\label{lem:exchange}
    Let $\Phi$ be a random smooth function on $\fR_{\K}$, then we have
    \begin{equation*}
        \E \sup_{\r:|\r|_{\infty}\leq M+2} \sup_{\Btheta\in\fR_{\K}} |\d^{\r}_{\Btheta} \Phi(\Btheta)|^{2n} \lesssim \sup_{\r:|\r|_{\infty}\leq M+3} \sup_{\Btheta\in\fR_{\K}} \E |\d^{\r}_{\Btheta} \Phi(\Btheta)|^{2n}
    \end{equation*}
    uniformly over $\K$.
\end{lem}
\begin{proof}
    Same as \cite[Lemma~4.3]{HX19}.
\end{proof}

In what follows, we fix a mollifier $\varrho$ on $\R$, let $\varrho_\delta := \delta^{-1} \varrho(\cdot / \delta)$, and write
\begin{equation*}
    F_\delta^{(\ell)} := F^{(\ell)} * \varrho_\delta\;, \qquad G_\delta^{(\ell)} := G^{(\ell)} * \varrho_\delta\;.
\end{equation*}
Let $P$ denote the heat kernel. For every function $\varphi$ and $x \in \R^d$, let
\begin{equation*}
    \varphi^{\lambda}_x(y) := \lambda^{-|\fs|} \varphi \Big(\frac{y-x}{\lambda}\Big)\;.
\end{equation*}
We write $\varphi^\lambda$ for $\varphi^\lambda_0$ for simplicity.

\subsection{The KPZ equation -- proof of Theorem~\ref{thm:kpz}}

In this section, we follow \cite[Section~5]{HX19}. We have
\begin{equ}
    \Psi_{\eps} = \d_x P \ast \xi_{\eps}, 
\end{equ}
where $\ast$ is space-time convolution and $\xi_{\eps}$ is the mollified space-time white noise as given in \eqref{e:kpz_macro}. \cite[Lemma~4.1]{HX19} deduces that $\Psi_{\eps}$ satisfies Assumption~\ref{as:cor} with $\alpha=1$. Here the kernel $K_0$ in \eqref{e:kernel} is a suitable truncation of $P$ and equals to $P$ in a domain containing the origin. Recall the regularity structure defined in \cite[Section~3.2]{HX19}. We first list all the symbols with their corresponding regularities appearing in the regularity structures.
\begin{equation} \label{e:table_kpz}
\begin{array}{ *{13}{c}}
	\toprule
	\text{object} \; (\tau): &\<2'>\; &\<2'1'>\; &\<1'>\; &\<2'2'0'>\; & \<2'2'0>\; &\<2'1'1'>\; &\<2'0'>\; &\<1'1'>\; &\<2'0>\; &\<0'>\; &\<2'1'0>\; &\<1'0>\\
	\midrule
	\text{reg.} \; (|\tau|): & -1- & -\frac{1}{2}- & -\frac{1}{2}- & 0- & 0- & 0- & 0- & 0- & 0- & 0- & \frac{1}{2}- & \frac{1}{2}-\\
	\bottomrule
\end{array}
\end{equation}

\begin{rmk}
    The ``regularity" $|\tau| = c-$ should be understood as $|\tau| = c-\kappa$ for sufficiently small $\kappa>0$. 

    Also, for the above objects, we call $\<0'>$, $\<1'>$, $\<2'>$, $\<1'0>$ and $\<2'0>$ first order processes. The processes $\<2'1'>$, $\<2'2'0>$, $\<2'0'>$, $\<1'1'>$ and $\<2'1'0>$ are built from two first order processes, and we call them second order processes. Similarly, $\<2'2'0'>$ and $\<2'1'1'>$ are third order processes. 
\end{rmk}

\begin{thm}
    \label{thm:sto_convergence}
    Let $\hPi^{\eps}$ be the renormalised model defined in \cite[(3.2),(3.3)]{HX19} and let $\Pi^{\KPZ}$ be the standard KPZ model defined in \cite[Appendix~A]{HX19}.  Then there exists $\zeta>0$ such that for every symbol $\tau$ in Table \eqref{e:table_kpz} with $|\tau|<0$, we have
    \begin{equation}
        \label{e:main_bound}
        \big( \E| \scal{\hPi_z^{\eps}\tau-\Pi_z^{\tiny\text{ KPZ}}\tau,\varphi^{\lambda}}|^{2n}\big)^{\frac{1}{2n}} \lesssim_n \eps^{\zeta} \lambda^{|\tau|+\zeta}
    \end{equation}
   uniformly in $\eps, \lambda \in(0,1)$, $z\in\R^{+}\times\T$ and smooth function $\varphi$ compactly supported in the unit ball with $\|\varphi\|_{\cC^1}\leq1$. As a consequence, $\hPi^{\eps}$ converges to $\Pi^{\KPZ}$ in probability in the space of modelled distributions. 
\end{thm}

\begin{rmk}
    By \cite[Proposition~6.3]{HQ}, the convergence of $\hPi^{\eps}$ is a direct corollary of \eqref{e:main_bound} with negative homogeneities. Then it suffices to prove \eqref{e:main_bound} for every symbol $\tau$ with $|\tau|<0$. Furthermore, we can assume $z=0$ by stationarity in the following proofs. 
\end{rmk}

    Now we are ready to prove Theorem~\ref{thm:kpz}.
    
\begin{proof}[Proof of Theorem~\ref{thm:kpz}]
    This follows directly from convergence of the stochastic objects (Theorem~\ref{thm:sto_convergence}), well-posedness and convergence of the abstract equation (Theorem~\ref{thm:abstract_eq}) and continuity of the reconstruction operator in the regularity structure. (See also the proof of \cite[Theorem 5.7]{HX19}). 
\end{proof}
    
    Now we turn to the proof of Theorem~\ref{thm:sto_convergence}. The first order processes except $\<0'>$ have been treated in \cite{Xu18}. $F \in \cC_{\pP}^{2,\beta}$ implies the desired decay of $\widehat{F}$ and $\hF'$. Then under parabolic metric, \cite[Theorem~1.4]{Xu18} shows that $\hPi^{\eps} \<2'>\rightarrow \Psi^{\diamond2}$ a.s. in $\cC^{-1-\kappa}$ and $\hPi^{\eps} \<1'>\rightarrow \Psi$ a.s. in $\cC^{-\frac{1}{2}-\kappa}$ by the bounds given in $\eqref{e:main_bound}$.

Theorem~\ref{thm:main_result} gives an improved bound for second order processes, which reduces the requirement for the regularity of $F$. The discussion will be found in Section~\ref{sec:spde_21}.
    
For the objects with regularity $0-$, we provide a simpler deterministic method. Write the term into the product $f\cdot g$ and choose proper mollified functions $f_{\delta}$ and $g_\delta$. Then \cite[Section~5]{HX19} has proved the desired convergence of the main part $f_\delta\cdot g_\delta$ since $f_\delta$ and $g_\delta$ are smooth. Now it suffices to show that the remaining parts converge to $0$. The convergence follows from the boundedness of $f_\delta,g_\delta$ and the smallness of the difference $f-f_\delta,g-g_\delta$. We provide details for the object \<2'1'1'> in Section~\ref{sec:reg_0} to demonstrate our method.

\subsubsection[]{The case \<0'>}
\label{sec:spde_0}
    Recall that we write $\tau_\eps$ for $\hPi^\eps \tau$. So we have
    \begin{equation*}
        \<0'>_\eps (z) = \frac{1}{2a} F''(\sqrt{\eps} \Psi_{\eps}(z)) - 1\;.
    \end{equation*}
    We split $\<0'>_\eps$ into
    \begin{equation*}
        \<0'>_\eps = \<0'>_\eps^{(\delta)} + \big( \<0'>_\eps - \<0'>_\eps^{(\delta)} \big)\;,
    \end{equation*}
    where $\<0'>_\eps^{(\delta)}$ is the process as $\<0'>_\eps$ but with $F''$ replaced by $F''_\delta$. Proceeding as other first order processes (see \cite[Section~5.2.1]{HX19}), one can show that almost surely
    \begin{equ}
        \bigg\| \frac{1}{2a} F_{\delta}''(\sqrt{\eps} \Psi_{\eps})-1 \bigg\|_{\cC^{-\kappa}} \lesssim \eps^{\kappa'} \delta^{-N} + \delta^\beta,
    \end{equ}
    where $N\in \N$ is a large constant, $\kappa, \kappa'$ are sufficiently small. Thus, as long as one chooses $\delta = \eps^{\nu}$ for $\nu>0$ sufficiently small, then
    \begin{equation*}
        \frac{1}{2a} F_{\delta}''(\sqrt{\eps} \Psi_{\eps}) \rightarrow 1
    \end{equation*}
    almost surely in $\cC^{-\kappa}$. For the remainder term $F''-F_{\delta}''$, we have
    \begin{equation*}
        |F''(x) - F_{\delta}''(x)| \lesssim \delta^{\beta} (1+|x|)^M\;.
    \end{equation*}
    Combining this and the fact that $\sqrt{\eps}\Psi_{\eps}$ is Gaussian with constant variance, we get the bound
    \begin{equation*}
        (\E |\scal{F''(\sqrt{\eps}\Psi_{\eps}) - F_{\delta}''(\sqrt{\eps}\Psi_{\eps}), \varphi^{\lambda}}|^{2n})^{\frac{1}{2n}} \lesssim \eps^{\nu \beta}.
    \end{equation*}
    Therefore, the bound \eqref{e:main_bound} for $\tau=\<0'>$ holds.

\subsubsection[]{The case \<2'1'>}
\label{sec:spde_21}

    In this case, $m_1=1$, $m_2=2$, and $\fF$ is defined by
    \begin{equ} \label{e:ff_21}
        \fF(\Btheta,x,y) = \tT_{(0)}(\sin(\theta_{\fx} X)) \tT_{(1)}(\cos(\theta_{\fy} Y)).
    \end{equ}
    Our aim is to show that $\<2'1'>_\eps$ converges to the object $\<2'1'>$ in the limiting KPZ model in $\cC^{-\frac{1}{2}-\kappa}$ in probability, where
    \begin{equation} \label{e:21}
        \<2'1'>_\eps(x) = \frac{1}{2 a^2 \eps^{\frac{3}{2}}} \int K(x,y) F'(X) \tT_{(1)}(F(Y)) dy\;.
    \end{equation}
    Again, we split $\<2'1'>_\eps$ into
    \begin{equation*}
        \<2'1'>_\eps = \<2'1'>_\eps^{(\delta)} + \big( \<2'1'>_\eps - \<2'1'>_\eps^{(\delta)} \big),
    \end{equation*}
    where $\<2'1'>_\eps^{(\delta)}$ is the same process as $\<2'1'>_\eps$ except that $F'$ and $F$ are replaced by their mollified versions $F'_\delta$ and $F_\delta$. 
    
    In \cite[Section~5.2]{HX19}, the authors showed that there exists $\nu>0$ such that as long as one chooses $\delta = \eps^{\nu}$, one has $\<2'1'>_\eps^{(\delta)} \rightarrow \<2'1'>$ in $\cC^{-\frac{1}{2}-\kappa}$ in probability. It remains to show the convergence of $\<2'1'>_\eps - \<2'1'>_\eps^{(\delta)}$ to $0$ with $\delta = \eps^{\nu}$ and $F \in \cC^{2,\eta}_\pP$. Fourier expanding $F$ and $F'$, we have
    \begin{equ}
        \scal{\<2'1'>_\eps - \<2'1'>_\eps^{(\delta)}, \varphi^\lambda} = (2a^2 \eps^{\frac{3}{2}})^{-1} \scal{ \widehat{F'_{\delta}} \otimes \widehat{F_{\delta}} - \widehat{F'} \otimes \widehat{F},\aA_{\eps,\lambda} \fF}_{\Btheta}\;,
    \end{equ}
    where the subscript $\Btheta$ means that the testing/integration is in the $\Btheta = (\theta_\fx, \theta_\fy)$ variable. By Lemmas~\ref{lem:local_decompose} and~\ref{lem:exchange}, we get
    \begin{equ}
        \| \scal{\<2'1'>_\eps - \<2'1'>_\eps^{(\delta)}, \varphi^\lambda} \|_{2n} \lesssim \sum_{\K\in\Z^2} \|\widehat{F'_{\delta}} \otimes \widehat{F_{\delta}} - \widehat{F'} \otimes \widehat{F}\|_{M+2,\fR_{\K}} \sup_{\r:|\r|_{\infty}\leq M+3} \sup_{\Btheta\in\fR_{\K}} \eps^{-\frac{3}{2}} \|\d_{\Btheta}^{\r} \aA_{\eps,\lambda}\fF(\Btheta)\|_{2n}.
    \end{equ}
    Choosing $\omega = \frac{\beta}{2}$ in Lemma~\ref{lem:local_bound} and plugging $\delta = \eps^{\nu}$, we have
    \begin{equ}
        \| \widehat{F'_{\delta}} \otimes \widehat{F_{\delta}} - \widehat{F'} \otimes  \widehat{F}\|_{M+2,\fR_{\K}} \lesssim \eps^{\frac{\beta \nu}{2}} (1+|K_1|)^{-1-\frac{\beta}{2}} (1+|K_2|)^{-2-\frac{\beta}{2}}\;.
    \end{equ}
    On the other hand, by Theorem~\ref{thm:main_result}, for every $\zeta>0$, we have the bound
    \begin{equation*}
           \sup_{\r:|\r|_{\infty}\leq M+3} \sup_{\Btheta\in\fR_{\K}} \eps^{-\frac{3}{2}} \big(\E |(\aA_{\eps,\lambda}\d_{\Btheta}^{\r}\fF)(\Btheta)|^{2n} \big)^{\frac{1}{2n}} \lesssim \eps^{-\zeta}\lambda^{-\frac{1}{2}-\kappa+\zeta}\;.
    \end{equation*}
    Choosing $\zeta < \frac{\beta \nu}{4}$, we obtain
    \begin{equation*}
        \| \scal{\<2'1'>_\eps - \<2'1'>_\eps^{(\delta)}, \varphi^\lambda} \|_{2n} \lesssim \eps^{\zeta} \lambda^{-\frac{1}{2}-\kappa}\;.
    \end{equation*}
    This finishes the proof of the case $\<2'1'>$.

\subsubsection[]{The case \<2'1'1'>} \label{sec:reg_0}
    Before approaching $\<2'1'1'>$, we first introduce two lemmas on the regularity of functions/distributions. We define a class of test functions
    \begin{equ}
        C_{p} = \{ \varphi\in\cC^{\infty}(\R^d)\ | \ \text{supp}(\varphi) \subset B(0,1),\ \|\varphi\|_{\cC^{p}} \leq 1\}.
    \end{equ}

\begin{lem} \label{lem:prod_holder}
    Let $p, q \in (0,1)$ such that $p>q$. Suppose $f \in \cC^p$ and $g \in \cC^{-q}$. Then for every $\eta>0$, there exists $C>0$ depending on $\eta$ such that
    \begin{equation*}
        \big|\scal{g, \big(f - f(x)\big) \varphi_x^\lambda}\big| \leq C \; \|f\|_{\cC^p} \|g\|_{\cC^{-q}} \; \lambda^{p-q-\eta}\;.
    \end{equation*}
    The bound is uniform over $\varphi \in C_p$, $x \in \R^d$ and $\lambda \in (0,1)$. 
\end{lem}
\begin{proof}
    Fix $\eta>0$, and let $\vartheta>0$ be sufficiently small depending on $\eta$ (and to be specified later). By duality, we have
    \begin{equation*}
        \big|\scal{g, \big(f - f(x)\big) \varphi_x^\lambda}\big| \leq \|g\|_{\cC^{-q}} \|\big(f - f(x) \big) \varphi_x^\lambda\|_{W^{q,1+\vartheta}}\;.
    \end{equation*}
    Since $f-f(x)$ is paired with $\varphi_x^\lambda$, we may assume without loss of generality that $f$ is supported within radius $\lambda$ from $x$. By the fractional Leibniz rule (see \cite[Theorem 1]{frac_Leibniz}), for any $s_1, s_2$ with $\frac{1}{s_1} + \frac{1}{s_2} = \frac{1}{1+\vartheta}$, we have
    \begin{equation*}
    \begin{split}
        \|\big(f - f(x) \big) \varphi_x^\lambda\|_{W^{q,1+\vartheta}} &\lesssim \|(f - f(x))\|_{W^{q,s_1}} \|\varphi_x^\lambda\|_{L^{s_2}} + \|(f-f(x))\|_{L^{s_1}} \|\varphi_x^\lambda\|_{W^{q,s_2}}\;\\
        &\lesssim \lambda^{p-q-\frac{|\fs| \vartheta}{1 + \vartheta}}\;.
    \end{split}
    \end{equation*}
    We then choose $\vartheta$ such that $\frac{|\fs| \vartheta}{1 + \vartheta} = \eta$. This completes the proof. 
\end{proof}

    The following lemma is also deterministic, and we omit the proof.
    \begin{lem} \label{lem:kernel_holder}
        Assume $-1<p<0$, $f\in\cC^p$ and $K(x,y)$ is a kernel with singularity $|\fs|-1$ and renormalisation constant $r_e=1$. Then for $g(x):= \int K(x,y)f(y) dy$, we have the bound
        \begin{equation*}
            |\scal{g,\varphi^\lambda}| \lesssim \lambda^{p+1} \| f \|_{\cC^p}
        \end{equation*}
        uniformly over $\varphi\in C_0$ and $\lambda\in(0,1)$.
    \end{lem}

    Recall that the process $\<2'1'1'>_\eps$ is given by
    \begin{equation*}
        \<2'1'1'>_\eps = \<2'1'0>_\eps \cdot \<1'>_\eps - C_{\eps}^{\<2'1'1's>}\;,
    \end{equation*}
    where
    \begin{equation} \label{e:210}
        \<2'1'0>_\eps (x) = \int K(x,y) \; \<2'1'>_\eps (y) dy \quad\text{and}\quad \<1'>_{\eps}(x) = \frac{1}{2a\sqrt{\eps}} F'(\sqrt{\eps}\Psi_{\eps}(x))\;,
    \end{equation}
    $K(x,y) = K_0(x-y) - K_0(-y)$, and $\<2'1'>_\eps$ is defined in \eqref{e:21}. Let $\<2'1'0>_\eps^{(\delta)}$ be the same in \eqref{e:210} except that $\<2'1'>_\eps$ is replaced by $\<2'1'>_\eps^{(\delta)}$ in the definition. Similarly, write $\<1'>_\eps^{(\delta)}$ for $\<1'>_\eps$ with $F'$ replaced by $F'_\delta$. The process $\<2'1'1'>_\eps$ can then be decomposed into three parts as
    \begin{equation*}
        \<2'1'1'>_\eps = \big( \<2'1'0>_\eps^{(\delta)} \cdot \<1'>_\eps^{(\delta)} - C_{\eps}^{\<2'1'1's>} \big) \; + \; \big( \<2'1'0>_\eps - \<2'1'0>_\eps^{(\delta)} \big) \cdot \<1'>_\eps \; + \; \<2'1'0>_\eps^{(\delta)} \; \big( \<1'>_\eps - \<1'>_\eps^{(\delta)} \big)\;.
    \end{equation*}
    In \cite[Section~5.2.2]{HX19}, the authors showed that there exists $\nu>0$ such that for $\delta = \eps^{\nu}$, one has the convergence
    \begin{equation*}
        \<2'1'0>_\eps^{(\delta)} \cdot \<1'>_\eps^{(\delta)} - C_{\eps}^{\<2'1'1's>} \rightarrow \<2'1'1'>
    \end{equation*}
    in probability in $\cC^{-\kappa}$. It remains to show the convergence in $\cC^{-\kappa}$ of the other two terms to $0$. This is where we use Lemma~\ref{lem:prod_holder} and the extra smallness from $\eps$ to turn the just-below-threshold ill-posed product into an analytically well-posed one. We give details for $\<2'1'0>_\eps^{(\delta)} \; \big( \<1'>_\eps - \<1'>_\eps^{(\delta)} \big)$, and the treatment for the other term is essentially the same. 

    In Section~\ref{sec:spde_21}, we have shown that for every $\zeta > 0$ and $n \in \N$, we have
    \begin{equation*}
        \big( \E \|\<2'1'>_\eps^{(\delta)}\|_{\cC^{-\frac{1}{2}-\zeta}}^{2n} \big)^{\frac{1}{2n}} \lesssim 1
    \end{equation*}
    for $\delta = \eps^\nu$, and uniformly in $\eps$. As a consequence of this bound and the regularising property of the kernel $K$, by Lemma \ref{lem:kernel_holder} we have
    \begin{equation}
    \label{e:210_bound}
        \big( \E \|\<2'1'0>_\eps^{(\delta)}\|_{\cC^{\frac{1}{2}-\zeta}}^{2n} \big)^{\frac{1}{2n}} \lesssim 1\;.
    \end{equation}
    We need the following lemma to control $\<1'>_\eps - \<1'>_\eps^{(\delta)}$ in a regularity space better than $-\frac{1}{2}$. 

    \begin{lem} \label{lem:J}
        There exists $\zeta'>0$ such that
        \begin{equ}
            (\E| \scal{\<1'>_\eps - \<1'>_\eps^{(\delta)}, \varphi^{\lambda}}|^{2n})^{\frac{1}{2n}} \lesssim_{n} \eps^{\zeta'} \lambda^{-\frac{1}{2} + \zeta'}.
        \end{equ}
        The proportionality constant is independent of $\eps,\lambda\in(0,1)$ and $\varphi\in C_0$.
    \end{lem}
    \begin{proof}
        First we consider the case $\eps\leq\lambda$. By \cite[Section~3]{Xu18}, we have
        \begin{equ}
            \|\scal{\<1'>_\eps - \<1'>_\eps^{(\delta)}, \varphi^{\lambda}}\|_{2n} \lesssim \sum_{K\in\Z} \|\widehat{F'}-\widehat{F'_{\delta}}\|_{M+2,\fR_{K}} \lambda^{-\frac{1}{2}-\eta}
        \end{equ}
        for any small $\eta>0$. Using Lemma~\ref{lem:local_bound} to control $\widehat{F'}-\widehat{F'_{\delta}}$, we get
        \begin{equ}
            \|\scal{\<1'>_\eps - \<1'>_\eps^{(\delta)}, \varphi^{\lambda}}\|_{2n} \lesssim \sqrt{\delta} \lambda^{-\frac{1}{2}-\eta}.
        \end{equ}
        Recall that $\delta=\eps^{\nu}$. If we choose $\eta = \frac{1}{6}\nu$, then the right hand side becomes $\eps^{\frac{\nu}{6}}\lambda^{-\frac{1}{2}+\frac{\nu}{6}}$ since $\eps\leq\lambda$. Now we turn to the case $\lambda\leq\eps$. By Lemma \ref{lem:local_decompose}, we have the bound
        \begin{equ}
            \|\scal{\<1'>_\eps - \<1'>_\eps^{(\delta)}, \varphi^{\lambda}}\|_{2n} \lesssim \eps^{-\frac{1}{2}} \sum_{K\in\Z} \|\widehat{F'}-\widehat{F'_{\delta}}\|_{M+2,\fR_{K}} \sup_{r\leq M+2} \sup_{\theta\in\fR_{K}} \|\d_{\theta}^{r} \bB_{\eps,\lambda} \fF(\theta)\|_{2n}\;,
        \end{equ}
        where $\bB_{\eps,\lambda} \fF(\theta)$ is defined by
        \begin{equ}
            \bB_{\eps,\lambda} \fF(\theta) = \int \sin (\theta X) \varphi^{\lambda}(x) dx\;.
        \end{equ}
        By using Lemma~\ref{lem:local_bound} to control $\widehat{F'}-\widehat{F'_{\delta}}$, the fact 
        \begin{equ}
            \|\d_{\theta}^{r} \bB_{\eps,\lambda} \fF(\theta)\|_{2n} \lesssim 1
        \end{equ}
        and the condition $\lambda\leq\eps$, we again obtain the same bound.
    \end{proof}
    
    Now we are ready to give the bound for $\<2'1'0>_\eps^{(\delta)} \big( \<1'>_\eps - \<1'>_\eps^{(\delta)} \big)$.
    
    \begin{prop} \label{prop:211_remainder}
        There exists a constant $\zeta>0$ such that
        \begin{equation*}
        \big( \E|\scal{\<1'>_\eps - \<1'>_\eps^{(\delta)}, \<2'1'0>_\eps^{(\delta)} \varphi^{\lambda}}|^{2n} \big)^{\frac{1}{2n}} \lesssim_n \eps^{\zeta}
        \end{equation*}
        holds uniformly over $\eps,\lambda\in(0,1)$ and $\varphi\in\cC_0$.
    \end{prop}
    \begin{proof}
        According to \eqref{e:210_bound}, for $\delta = \eps^\nu$, the bound
        \begin{equation*}
            \E \big( \| \<2'1'0>_\eps^{(\delta)} \|_{\cC^{ \frac{1}{2} - \frac{\nu}{20}}}^{4n}\big)^{\frac{1}{4n}} \lesssim 1  
        \end{equation*}
        holds uniformly over $\eps\in(0,1)$. And Lemma~\ref{lem:J} implies the Kolmogorov type bound
        \begin{equation*}
            \E \big( \| \<1'>_\eps - \<1'>_\eps^{(\delta)}\|_{\cC^{ -\frac{1}{2} + \frac{\nu}{10}}}^{4n}\big)^{\frac{1}{4n}} \lesssim \eps^{\frac{\nu}{10}}  
        \end{equation*}
        uniformly over $\eps\in(0,1)$.
        By Lemma~\ref{lem:prod_holder} we have
        \begin{equation*}
        \big( \E|\scal{\<1'>_\eps - \<1'>_\eps^{(\delta)}, \<2'1'0>_\eps^{(\delta)}  \varphi^{\lambda}}|^{2n} \big)^{\frac{1}{2n}} \lesssim \lambda^{\frac{\nu}{30}} \left( \E(\| \<1'>_\eps - \<1'>_\eps^{(\delta)}\|_{\cC^{ -\frac{1}{2} + \frac{\nu}{10}}} \| \<2'1'0>_\eps^{(\delta)}\|_{\cC^{ \frac{1}{2} - \frac{\nu}{20}}} )^{2n} \right)^{\frac{1}{2n}}
        \end{equation*}
        Then the conclusion holds by H\"older inequality.
    \end{proof}

    The bound for the remaining term $\big( \<2'1'0>_\eps - \<2'1'0>_\eps^{(\delta)} \big) \cdot \<1'>_\eps$ can be obtained in essentially the same way. Thus, we obtain the convergence of $\<2'1'1'>_\eps$ to $\<2'1'1'>$.

    \subsection{Dynamical \texorpdfstring{$\Phi^4_3$}{F} -- proof of Theorem~\ref{thm:phi43}}

	First we introduce the regularity structure in the dynamical $\Phi_3^4$ model. Let $\iI$ denote the abstract integration map corresponding to the heat kernel. We then use graphical notations by setting
	\begin{equ}
		\<3'0> = \iI(\<3'>)\;, \quad \<2'2'> = \<2'>\cdot\iI(\<2'>)\;, \quad  \<3'1'> = \<1'>\cdot\<3'0>\;,\quad  \<3'2'> = \<2'>\cdot\<3'0>\;. 
	\end{equ}
	We choose a proper coupling constant $a$ and define $\Psi_{\eps} := P * \xi_{\eps}$ where $P$ denotes the heat kernel and $\xi_\eps$ is given in \eqref{e:phi43_macro}. For every $\eps>0$, we 
	define a representation $\Pi^\eps$ of the regularity structure by
	\begin{equation*}
		\begin{split}
			(\Pi^{\eps} \<0'>)(z) &= \frac{1}{6 a} G^{(3)}\big(\eps^{\frac{1}{2}} \Psi_\eps(z)\big) - 1, \quad (\PPi^{\eps} \<1'>)(z) = \frac{1}{6 a \sqrt{\eps}} G''\big(\eps^{\frac{1}{2}} \Psi_\eps(z)\big), \\
			(\Pi^{\eps} \<2'>)(z) &= \frac{1}{3 a \eps} G'\big(\eps^{\frac{1}{2}} \Psi_\eps(z)\big) - C_{\<2's>}^{(\eps)}, \\
           (\Pi^{\eps} \<3'>)(z) &= \frac{1}{a \eps^{\frac{3}{2}}} G\big(\eps^{\frac{1}{2}} \Psi_\eps(z)\big) - 3C_{\<2's>}^{(\eps)}\Psi_\eps(z),
		\end{split}
	\end{equation*}
	where the constant $C_{\<2's>}^{(\eps)}$ is chosen to satisfy $\E(\Pi^{\eps} \<2'>)(z) = 0$. And we set $\hPi^{\eps} \tau = \Pi^{\eps} \tau$ for $\tau \in \{\<0'>,\<1'>,\<2'>,\<3'>\}$ and
	\begin{equation} \label{e:model_2}
		\begin{split}
			(\hPi^{\eps} \<2'2'>)(z) &= (\hPi^{\eps} \<2'0>)(z)\cdot (\hPi^{\eps} \<2'>)(z) - C_{\<2'2's>}^{(\eps)}\;, \\
           (\hPi^{\eps} \<3'1'>)(z) &= (\hPi^{\eps} \<3'0>)(z) \cdot (\hPi^{\eps} \<1'>)(z) - C_{\<3'1's>}^{(\eps)}\;, \\
			(\hPi^{\eps} \<3'2'>)(z) &= (\hPi^{\eps} \<3'0>)(z) \cdot (\hPi^{\eps} \<2'>)(z) - ( 3C_{\<2'2's>}^{(\eps)} + 2C_{\<3'1's>}^{(\eps)} )\Psi_\eps(z)\;,
		\end{split}
	\end{equation}
	where the constants $C_{\tau}^{(\eps)}$ are chosen to satisfy $\E (\hPi^{\eps} \tau)(0) = 0$ for all objects $\tau$ appearing in \eqref{e:model_2}. As before, we list all the symbols with their corresponding regularities.
	\begin{equation}
	\label{e:table_phi43}
	\begin{array}{ *{8}{c}}
	    \toprule
		\text{object} \; (\tau) &\<0'>\; &\<1'>\; &\<2'>\; &\<3'0>\; & \<3'1'>\; &\<2'2'>\; &\<3'2'>\\
		\midrule
		\text{reg.} \; (|\tau|) & 0- & -\frac{1}{2}- & -1- & \frac{1}{2}- & 0- & 0- & -\frac{1}{2}- \\
		\bottomrule
	\end{array}
    \end{equation}
    
    \begin{thm}
    \label{thm:phi4_sto_convergence}
    Let $\hPi^{\eps}$ be the renormalised model given above and let $\Pi^{\Phi}$ be the standard dynamical $\Phi_3^4$ model given in \cite[Section~10.5]{rs_theory}. Then there exists $\zeta>0$ such that for every symbol $\tau$ in Table~\eqref{e:table_phi43} with $|\tau|<0$, we have
    \begin{equation}
        \label{e:phi4_main_bound}
        \big( \E| \scal{\hPi_z^{\eps}\tau-\Pi_z^{\Phi}\tau,\varphi^{\lambda}}|^{2n}\big)^{\frac{1}{2n}} \lesssim_n \eps^{\zeta} \lambda^{|\tau|+\zeta},
    \end{equation}
    where the bound holds uniformly in $\eps\in(0,1)$, $\lambda\in(0,1)$, $z\in\R^{+}\times\T^3$ and smooth function $\varphi$ compactly supported in $\{|x|<1\}$. As a consequence, $\hPi^{\eps}$ converges to $\Pi^{\Phi}$ in probability in the space of modelled distributions.
\end{thm}
    
    As before, Theorem~\ref{thm:phi43} follows from Theorems~\ref{thm:abstract_eq} and~\ref{thm:phi4_sto_convergence}. Then it remains to prove Theorem~\ref{thm:phi4_sto_convergence}. We again write $\tau_\eps$ for $\hPi^\eps \tau$, and $\tau$ for $\Phi^\Phi \tau$. The first order processes can be treated in the same way as shown in \cite{Xu18} and Section~\ref{sec:spde_0}. And we adopt the same procedure as in Section~\ref{sec:reg_0} to deal with the objects (\<3'1'>,\;\<2'2'>) with regularity $0-$. Then it remains to focus on the object \<3'2'>.
    
    \begin{rmk}
        The way to prove convergence for $\<3'2'>$ is still the same: decomposing it into $\<3'2'>_\eps^{(\delta)}$ and a small remainder, applying \cite{HX19} to prove the convergence of $\<3'2'>_\eps^{(\delta)}$ to $\<3'2'>$, and Theorem~\ref{thm:main_result} to prove the convergence of the remainder to $0$. 

        For the convergence of $\<3'2'>_\eps^{(\delta)}$, one notices that the form in \cite[Theorem~6.2]{HX19} only has the $0$-th chaos removed in the Wick ordering part, while $\<3'>$ has its first chaos removed. But the proof of \cite[Theorem~6.2]{HX19} did not make use of the precise number of removed chaos in Wick ordering. Making it arbitrary only affects the power of $|\Btheta|$ in the bound, but does not change its form. Hence, one can still apply it to obtain convergence of $\<3'2'>_\eps^{(\delta)}$ to $\<3'2'>$ with $\delta = \eps^{\nu}$ for some sufficiently small $\nu$. 
    \end{rmk}
    
    In the case $\tau=\<3'2'>$, we have $m_1=2,m_2=3,\alpha=1$, and $\fF$ is given by
    \begin{equ}
        \fF(\Btheta,x,y) = \tT_{(1)}(\cos (\theta_{\fx} X)) \tT_{(2)}(\sin (\theta_{\fy} Y)).
    \end{equ}
    Proceeding as in Section~\ref{sec:spde_21}, it suffices to prove the following statement.
    \begin{prop}
        Recall the operator $\aA_{\eps,\lambda}$ defined in \eqref{e:operator}. We have the bound
        \begin{equ}
            \eps^{-\frac{5}{2}}\E |(\aA_{\eps,\lambda}\d_{\Btheta}^{\r}\fF)(\Btheta)|^{2n} \lesssim \eps^{-\zeta}\lambda^{-\frac{1}{2}-\kappa+\zeta}
        \end{equ}
        for sufficiently small $\zeta>0$. And the bound holds uniformly in $\eps,\lambda\in(0,1)$.
    \end{prop}
    \begin{proof}
        Note that the kernel $K$ has singularity $5-2$ at the origin and positive renormalisation degree $r_e=1$. Thus the result is a direct conclusion of Theorem~\ref{thm:main_result}.
    \end{proof}
    Then we can proceed as \cite[Section~5]{HX19} to conclude the convergence of $\<3'2'>_{\eps}$ to $\<3'2'>$ in the limiting dynamical $\Phi^4_3$ model.

\appendix
\section{Proof of the correlation bounds in Section~\ref{sec:correlation_bounds}} \label{sec:singleton_proof}

In this appendix, we prove Lemmas~\ref{le:comparable_bound}, ~\ref{le:singleton_bound} and~\ref{le:fixed_bound} from Section~\ref{sec:correlation_bounds}. We write in a more general setting. Let $K$ be a positive integer\footnote{We use this notation since the integration kernel will not appear in the appendix.}, and $\vec{z} = (z_1, \dots, z_{K})$ be a collection of $K$ space-time points. For each $j$, let $Z_{j} = \eps^{\frac{\alpha}{2}} \Psi_\eps(z_j)$. Also let $\Btheta = (\theta_j)_{j \in [K]} \in \R^{K}$. We aim to bound the quantity
\begin{equation} \label{e:trig_general_form}
    \E \prod_{j=1}^{K} \d_{\theta_j}^{r_j} \tT_{(t_j-1)} \big( \trig_{\eta_j} (\theta_j Z_{j}) \big)
\end{equation}
in terms of multi-point correlations of $Z_j$'s, and being uniform in $\eps$, the frequencies $\Btheta$, and the locations $\vec{z}$. Lemmas~\ref{le:comparable_bound} and~\ref{le:singleton_bound} correspond to $K=4n$, $z_j = x_j$, $\theta_j = \theta_{\fx}$, $\eta_j = \zeta_1$ and $t_j = m_1$ for $j \leq 2n$, and $z_{j} = y_{j-2n}$, $\theta_j = \theta_{\fy}$, $\eta_j=\zeta_2$ and $t_j = m_2$ for $j \geq 2n+1$. As for Lemma~\ref{le:fixed_bound}, we will split the left hand side of \eqref{e:fixed_bound_alpha} into several terms, each of which is of the form \eqref{e:trig_general_form} with $K = 2n+1$, $z_j = y_j$, $\eta_j = \zeta_1$ for $j \leq 2n$ and $z_{2n+1} = x$, $\eta_{2n+1} = \zeta_2$. 

The strategy of the proof of the correlation bounds mainly follows the clustering arguments developed in \cite[Section 6]{HX19} and its refinement in \cite[Section 2]{Xu18}. Fix $\Btheta$ and $\vec{z}$ (and we aim to obtain bounds independent of them).

Let $L_0$ be a sufficiently large constant whose value (depends on $n$, $\r$ and $\Lambda$ but independent of $\eps$, $\lambda$ and $\Btheta$) will be specified later (see the proof of Lemma~\ref{le:comparable_bound} below). Let $\Clus$ denote the equivalence class of $[K]$ obtained by the equivalence relation $\sim$ such that $j \sim j'$ if and only if there exists $k \leq K-1$ and $j_0, \dots, j_k \in [K]$ such that $|z_{j_{i+1}} - z_{j_i}| \leq L_{0} \eps$ for every $i=0, \dots, k-1$. Note that our clustering parameter here is different from the parameter in Section $\ref{sec:proof}$, where we only do clustering with respect to some of the points rather than all points. This subtlety would only appear in the proof of Lemma \ref{le:singleton_bound} where we let $L = 3n L_0$. Let
\begin{equation*}
    \sS = \{\set \in \Clus: |\set| = 1\}
\end{equation*}
be the set of singletons in $\Clus$, and $\uU := \Clus \setminus \sS$. With an abuse of notation, we will write $s \in \sS$ instead of $\{s\} \in \sS$.

The following proposition is implied by \cite[eq.(2.5)]{Xu18}.
\begin{prop} \label{pr:no_singleton}
    We have the bound
    \begin{equation*}
        \E \prod_{j=1}^{K} \Big( \sum_{k = t_j}^{t+1} Z_{j}^{\diamond k} \Big) \gtrsim 1\;,
    \end{equation*}
    where $t = \max_{j \in [K]} t_j$. The bound is uniform over $\eps \in (0,1)$, $\Btheta \in \R^{K}$ and point configuration $\vec{z}$ such that $|\sS| = \emptyset$. 
\end{prop}

Since the quantity \eqref{e:trig_general_form} is bounded by a constant depending on $K$, $\Lambda$, $\r$ and $\t$ only, Proposition~\ref{pr:no_singleton} gives the bound
\begin{equation*}
    \Big| \E \prod_{j=1}^{K} \d_{\theta_j}^{r_j} \tT_{(t_j - 1)} \big( \trig_{\eta_j} (\theta_j Z_{j}) \big) \Big| \lesssim \E \prod_{j=1}^{K} \Big( \sum_{k = t_j}^{t+1} Z_{j}^{\diamond k} \Big)\;,
\end{equation*}
uniformly over all frequencies $\Btheta \in \R^K$ and all point configurations $\vec{z}$ with the constraint that $\sS = \emptyset$. This already matches Lemmas~\ref{le:comparable_bound} and~\ref{le:singleton_bound} for point configurations with $\sS = \emptyset$, and ``almost" matches Lemma~\ref{le:fixed_bound} (this issue will be addressed below). Hence, we now focus on the point configurations with $\sS \neq \emptyset$. 


Suppose $|\sS| \geq 1$. Following the notations in \cite[Section 6]{HX19} and \cite[Section 2]{Xu18}, for $\Z=(Z_j)_{j\in[K]}$,  $\Btheta=(\theta_j)_{j\in[K]}$ and $\n=(n_j)_{j\in[K]}$, we write $\Z_{\set}$, $\Btheta_\set$ and $\n_{\set}$ for restrictions on a cluster $\set\in\Clus$, and also $\Z_{\set}^{\diamond\n_{\set}} = \bdiamond_{j\in\set}Z_{j}^{\diamond n_j}$. Then we can write
\begin{equation*}
    \E \prod_{j=1}^{K} \d_{\theta_j}^{r_{j}} \tT_{(t_j-1)}(\trig_{\eta_j}(\theta_j Z_{j})) = \sum_{N\geq0} \sum_{\stackrel{\n\in\N^{K}}{|\n|=N}} \bigg( \prod_{\set\in\Clus} C_{\n_\set}(\Btheta_\set,\Z_{\set}) \bigg) \bigg( \E\prod_{\set\in\Clus} \Z_{\set}^{\diamond\n_{\set}} \bigg),
\end{equation*}
where $C_{\n_{\set}}(\Btheta_{\set},\Z_{\set})$ is the coefficient of $\Z_{\set}^{\diamond \n_\set}$ in the chaos expansion of
\begin{equation*}
    \prod_{j \in \set} \d_{\theta_j}^{r_j} \tT_{(t_j-1)} \big( \trig_{\eta_j} (\theta_j Z_{j}) \big)\;,
\end{equation*}
and has the expression
\begin{equation} \label{e:chaos_coefficient_formula}
    C_{\n_{\set}}(\Btheta_{\set},\Z_{\set}) = \frac{1}{\n_{\set}!} \cdot \E\prod_{j\in\set} \left(\d_{\theta_j}^{r_{j}} \d_{Z_{j}}^{n_{j}} \tT_{(t_j-1)}( \trig_{\eta_j}(\theta_j Z_{j}))\right).
\end{equation}
Note that for $s \in \sS$, the chaos expansion for $Z_s$ starts from $Z_{s}^{\diamond t_s}$. Hence, we can bound the above quantity by
\begin{equation} \label{e:trig_expansion_bound}
    \begin{split}
    &\phantom{1111}\bigg| \E \prod_{j=1}^{K} \d_{\theta_j}^{r_{j}} \tT_{(t_j-1)}(\trig_{\eta_j}(\theta_j Z_{j})) \bigg|\\
    &\leq \sum_{N \geq 0} \bigg[ \bigg( \sum_{\stackrel{|\n| = N}{\n_\sS \geq \t_\sS}} \prod_{\set \in \Clus} \big| C_{\n_\set} (\Btheta_\set, \Z_\set) \big| \bigg) \times \sup_{\stackrel{|\n| = N}{\n_\sS \geq \t_\sS}} \E \Big( \prod_{\set \in \Clus} \Z_{\set}^{\diamond \n_\set} \Big) \bigg]\;,
    \end{split}
\end{equation}
where both the sum and supremum are taken over $\n_\sS \geq \t_\sS$ in the sense that $n_s \geq t_s$ for every $s \in \sS$. 

Although the following lemma is not explicitly stated in \cite[Section~2]{Xu18}, it is a direct corollary of \cite[Proposition~2.3, 2.8, 2.9 and Section~2.6]{Xu18}

\begin{lem} \label{le:reduction_enhancement}
    There exists $C>0$ depending on $K$, $\Lambda$ and $\r$ only such that
    \begin{equation*}
        \sup_{\stackrel{|\n| = N}{\n_\sS \geq \t_\sS}} \E \prod_{\set \in \Clus} \Z_{\set}^{\diamond \n_\set} \leq N!! \cdot \big( C / L_{0}^{\frac{\alpha}{2}}\big)^{N} \cdot \E \prod_{j=1}^{K} \Big( \sum_{k_j = t_j}^{t+1} Z_{j}^{\diamond k_j} \Big)\;,
    \end{equation*}
    uniformly over $\eps \in (0,1)$ and the location of $\vec{z}$. 
\end{lem}

The correlation on the right hand side above already matches that in Lemmas~\ref{le:comparable_bound} and~\ref{le:singleton_bound}, and almost matches that in Lemma~\ref{le:fixed_bound}. It now remains to control the sum of the coefficients $C_{\n_\set}(\Btheta_\set, \Z_\set)$. The following lemma is the same as \cite[Lemma~2.2]{Xu18}. 

\begin{lem} \label{le:coefficients_bound_1}
    There exists $C>0$ depending on $K$, $\Lambda$ and $\r$ only such that
    \begin{equation*}
        |C_{\n_\set}(\Btheta_\set, \Z_\set)| \leq \frac{\big( C(1+|\Btheta|) \big)^{|\n_\set|}}{\n_\set !}\;.
    \end{equation*}
    Furthermore, we have
    \begin{equation*}
        \sum_{|\n| = N + |\t_\sS|} \prod_{\set \in \Clus} \big| C_{\n_\set} (\Btheta_\set, \Z_\set) \big| \leq e^{-\frac{1}{2\Lambda} \sum_{s \in \sS} \theta_s^2} \cdot \frac{\big( C(1+|\Btheta|) \big)^N}{N!}\;.
    \end{equation*}
    Both bounds are uniform in $\Btheta$. 
\end{lem}

The above lemma is sufficient to prove Lemma~\ref{le:comparable_bound} since it gives a Gaussian decay in terms of the frequencies from the singleton sets. 
The Gaussian decay can also be obtained if there exists a cluster with a unique point in it, whose corresponding frequency is much larger than the frequencies of the other points in this cluster. This phenomena is precisely stated in the following lemma and is of great importance in the proof of Lemma \ref{le:singleton_bound} and Lemma \ref{le:fixed_bound}.

\begin{lem} \label{le:coefficients_bound_2}
    Let $A \subset [K]$ and $\ell \in A$ be such that
    \begin{equation*}
        |\theta_\ell| > 4 |A| \Lambda^2 \sum_{j \in A \setminus \{\ell\}} |\theta_j|\;.
    \end{equation*}
    Let $C_{\n_A}(\Btheta_A, \Z_A)$ be the coefficient of $\Z_A^{\diamond \n_A}$ in the chaos expansion of
    \begin{equation*}
        \prod_{j \in A} \d_{\theta_j}^{r_j} \tT_{(t_j - 1)} \big( \trig_{\eta_j}(\theta_j Z_j) \big)\;.
    \end{equation*}
    Then there exists $C>0$ depending on $|A|$, $\Lambda$, $\r$ and $\t$ only such that
    \begin{equation*}
        |C_{\n_A}(\Btheta_A, \Z_A)| \leq e^{-\frac{|\theta_\ell|^2}{4 \Lambda}} \cdot \frac{\big( C(1 + |\Btheta_A|)\big)^{|\n_A|}}{\n_A !}\;.
    \end{equation*}
\end{lem}
\begin{proof}
    We rewrite the coefficient as
        \begin{equation*}
            C_{\n_{A}}(\Btheta_A,\Z_A) = \frac{1}{\n_A !} \E \prod_{j\in A}  \d_{\theta_{j}}^{r_{j}} \Big(\theta_{j}^{n_{j}} \tT_{(t_j-1-n_j)} (\trig_{\eta_j}(\theta_{j} Z_{j})) \Big)\;,
        \end{equation*}
        where $\tT_{t} = \id$ if $t<0$. 
        Distributing $r_{j}$ derivatives of $\theta_j$ into the two terms, the differentiation of $\theta_{j}^{n_j}$ yields an factor which is at most $n_{j}^{r_j}$. Then the product of these factors is bounded by $C^{|\n_A|\max_{j\in A} \{r_j\}}$ where $C$ is a constant. A typical part of the differentiation of the second term is $\E \prod_{j\in A}  \Xi_j(Z_j)$, where the function $\Xi_j(w):= w^{q^{(1)}_j} \trig_{\eta_j}(\theta_{j} w)$ or $\  e^{-\frac{\theta_j^2 \E w^2}{2}} w^{q^{(2)}_j}$,
        where the constant $q_j^{(1)},q_j^{(2)}$ are bounded by $t_j + r_j$. 
        We distinguish two cases by the choice of $\Xi_{\ell}$. For the case $\Xi_{\ell} (w)= e^{-\frac{\theta_{\ell}^2 \E w^2}{2}} w^{q^{(2)}_{\ell}}$, since $|\Xi_j(w)|\leq 1 + |w|^{t_j+r_j}$ and the Gaussianity of $Z_j$, we obtain the bound
        \begin{equation*}
            \Big| \E \prod_{j\in A}  \Xi_j(Z_j) \Big| \lesssim e^{-\frac{|\theta_{\ell}|^2}{2\Lambda}}.
        \end{equation*}
        For the case $\Xi_{\ell}(w) = w^{q^{(1)}_{\ell}} \trig_{\eta_{\ell}}(\theta_{\ell} w)$, it is easy to check that $\prod_{j\in A}  \Xi_j(Z_j)$ is a sum of at most $2^{|A|}$ terms with the form 
        \begin{equation*}
            \Big( \prod_{j\in A} Z_j^{q_j^{(k_j)}} \Big) \cdot \trig_{\pm}\Big( a_{\ell} \theta_{\ell} Z_{\ell} +\sum_{j\in A\setminus\{\ell\}} a_j \theta_j Z_j \Big),
        \end{equation*}
        where $a_{\ell}\in\{-1,1\}$, $a_{j}\in\{-1,0,1\}$ and $k_{j}\in\{1,2\}$ for $j\in A\setminus\{\ell\}$. Gaussianity implies that
        \begin{equation*}
        \begin{split}
            &\Big|\E\Big[\Big( \prod_{j\in A} Z_j^{q_j^{(k_j)}} \Big) \trig_{\pm}\Big( a_{\ell} \theta_{\ell} Z_{\ell} +\sum_{j\in A\setminus\{\ell\}} a_j \theta_j Z_j \Big)\Big] \Big|\\ \lesssim & (1+|\theta_{\ell}|)^{|A| \max\limits_{j\in A} (t_j+r_j)} \exp \bigg( -\frac{\E\big( a_{\ell} \theta_{\ell} Z_{\ell} +\sum_{j\in A\setminus\{\ell\}} a_j \theta_j Z_j \big)^2}{2} \bigg)\\
            \lesssim &\exp \bigg( -c_1 |\theta_{\ell}|^2 + c_2 |\theta_{\ell}| \Big( \sum_{j\in A\setminus\{\ell\}} |\theta_j| \Big) +c_3 \Big( \sum_{j\in A\setminus\{\ell\}} |\theta_j| \Big)^2 \bigg)\;,
        \end{split}
        \end{equation*}
        where $c_1 \geq \frac{7}{16\Lambda},\;|c_2| \leq \Lambda$ and $|c_3|\leq \Lambda$. Note that
        \begin{equation*}
            |\theta_{\ell}| > 4|A| \Lambda^2 \sum_{j\in A\setminus\{\ell\}} |\theta_j|\;,
        \end{equation*}
        then we obtain the bound for the typical part
        \begin{equ}
            \Big| \E \prod_{j\in A\setminus \{\ell\}}  \Xi_j(Z_j) \Big| \lesssim e^{-\frac{|\theta_{\ell}|^2}{4\Lambda}}.
        \end{equ}
        Thus we complete the proof.
\end{proof}

We are now ready to complete the proofs of Lemmas~\ref{le:comparable_bound},~\ref{le:singleton_bound} and~\ref{le:fixed_bound}. For Lemmas~\ref{le:comparable_bound} and~\ref{le:singleton_bound}, we have $K=4n$, and for $1 \leq j \leq 2n$, we have $Z_j = X_j$, $\theta_j = \theta_{\fx}$, and $t_j = m_1$, while for $2n+1 \leq j \leq 4n$, we have $Z_j = Y_{j-2n}$, $\theta_j = \theta_{\fy}$, and $t_j = m_2$.

\begin{proof} [Proof of Lemma~\ref{le:comparable_bound}]
    Applying Lemmas~\ref{le:reduction_enhancement} and~\ref{le:coefficients_bound_1} to the right hand side of \eqref{e:trig_expansion_bound}, and using the assumption that $\theta_{\fx}$ and $\theta_{\fy}$ are comparable, we see there exists $C > 0$ depending on $n$, $\Lambda$ and $\r$ only such that
    \begin{equation} \label{e:trig_bound_general}
        \begin{split}
        &\phantom{111}\bigg| \E \prod_{j=1}^{2n} \d_{\theta_{\fx}, \theta_{\fy}}^{\r} \fF(\theta_{\fx}, \theta_{\fy}, x_j, y_j) \bigg|\\
        &\leq \exp \Big(-\frac{\theta_{\fx}^2}{C} + \frac{C(1+|\theta_{\fx}|)^2}{L_{0}^\alpha}\Big) \cdot \sum_{k_1 = m_1}^{(m_1 \vee m_2)+1} \sum_{k_2 = m_2}^{(m_1 \vee m_2)+1} \E \prod_{j=1}^{2n} \big( X_{j}^{\diamond k_1} Y_{j}^{\diamond k_2} \big)\;.
        \end{split}
    \end{equation}
    By choosing $L_{0}$ sufficiently large (depending on $n$, $\Lambda$ and $\r$ only), we can guarantee that the exponential term on the right hand side of \eqref{e:trig_bound_general} is uniformly bounded in $(\theta_{\fx}, \theta_{\fy})$. This completes the proof of Lemma~\ref{le:comparable_bound}. 
\end{proof}

The proof for Lemma~\ref{le:singleton_bound} is similar, except using Lemma~\ref{le:coefficients_bound_2} instead of Lemma~\ref{le:coefficients_bound_1}. 

\begin{proof} [Proof of Lemma~\ref{le:singleton_bound}]
    We prove \eqref{e:comparable_bound} with the assumption $|\theta_{\fx}| > 100 n (1 + \Lambda^2) |\theta_{\fy}|$ and $\vec{x} \in \sS_{2n}$. The case with $|\theta_{\fy}| > 100 n (1 + \Lambda^2) |\theta_{\fx}|$ and $\vec{y} \in \sS_{2n}$ is identical. 

    By definition of $\sS_{2n}$ and that $L = 3n L_0$, we know there exists $\ell \in \{1, \dots, 2n\}$ and $\set^* \in \Clus$ such that
    \begin{equation*}
        \set^* \cap \{1, \dots, 2n\} = \{\ell\}\;.
    \end{equation*}
    In other words, $x_{\ell}$ is in the cluster indexed by $\set^*$ but no other $x_j$ is if $j \neq \ell$. Since $|\theta_{\fx}| > 100 n (1 + \Lambda^2) |\theta_{\fy}|$, by Lemma~\ref{le:coefficients_bound_2}, we have
    \begin{equation*}
        |C_{\n_{\set^*}}(\Btheta_{\set^*}, \Z_{\set^*})| \leq e^{-\frac{\theta_{\fx}^2}{4 \Lambda}} \cdot \frac{\big( C(1 + |\theta_{\fx}|) \big)^{|\n_{\set^*}|}}{\n_{\set^*}!}\;.
    \end{equation*}
    As a consequence, we have
    \begin{equation*}
        \sum_{|\n| = N} \prod_{\set \in \Clus} |C_{\n_{\set}}(\Btheta_{\set}, \Z_{\set})| \leq e^{-\frac{\theta_{\fx}^2}{4 \Lambda}} \cdot \frac{\big( C (1+|\theta_{\fx}|) \big)^{N}}{N!}\;.
    \end{equation*}
    Applying the above bound together with Lemmas~\ref{le:coefficients_bound_2} to the right hand side of \eqref{e:trig_expansion_bound}, we get the same bound as the right hand side of \eqref{e:trig_bound_general}. One can proceed as before to choose $L_0$ sufficiently large to conclude Lemma~\ref{le:singleton_bound}. 
\end{proof}

We now turn to Lemma~\ref{le:fixed_bound}. First we write the left hand side of \eqref{e:fixed_bound_alpha} as a linear combination of terms which are of the form \eqref{e:trig_general_form} with $K=2n+1$. In fact, $\big( \d_{\theta_{\fx}}^{r_1} \tT_{m_1-1} \big( \trig_{\zeta_1}(\theta_{\fx} X) \big) \big)^{2n}$ is a linear combination of terms of the form $X^{\ell} \trig_{\pm}(p \theta_{\fx} X)$ where $p \in \{-2n, \dots, 0, \dots, 2n\}$, $\ell$ ranges over non-negative integers with a fixed upper bound depending on $\r$ and $n$ only, and the coefficients in the linear combination are all independent of $\theta_{\fx}$. Since $X^{\ell} \trig_{\pm}(p \theta_{\fx} X)$ is proportional to $\d_{p \theta_{\fx}}^{\ell} \trig_{\pm}(p \theta_{\fx} X)$, it now suffices to control the quantity
\begin{equation} \label{e:aim_quantity}
    \E \Big( \d_{p \theta_{\fx}}^{\ell} \trig_{\pm}(p \theta_{\fx} X) \; \prod_{j=1}^{2n} \d_{\theta_{\fy}}^{r_{2}} \tT_{(m_2 - 1)} \big( \trig_{\zeta_2}(\theta_{\fy} Y_j) \big) \Big)\;.
\end{equation}
This is in the form of \eqref{e:trig_general_form} with $K=2n+1$, $z_j = y_j$ and $\theta_j = \theta_{\fy}$ for $j \leq 2n$, and $z_{2n+1} = x$ and $\theta_{2n+1} = p \theta_{\fx}$. We are now ready to complete the proof. 

\begin{proof} [Proof of Lemma~\ref{le:fixed_bound}]
We carry out the proof in two steps. We first show the bound
\begin{equation} \label{e:fixed_first_bound}
    \begin{split}
    &\phantom{111}\E \Big( \d_{p \theta_{\fx}}^{\ell} \trig_{\pm}(p \theta_{\fx} X) \; \prod_{j=1}^{2n} \d_{\theta_{\fy}}^{r_2} \tT_{(m_2 - 1)} \big( \trig_{\zeta_2}(\theta_{\fy} Y_j) \big) \Big)\\
    &\lesssim \E \bigg[ \Big( \sum_{k=0}^{(m_1\vee m_2)+1}X^{\diamond k} \Big) \Big( \prod_{i=1}^{2n}  \sum_{k=m_2}^{(m_1\vee m_2)+1}Y_i^{\diamond k}\Big) \bigg]\;,
    \end{split}
\end{equation}
and then control the right hand side of \eqref{e:fixed_first_bound} by that of \eqref{e:fixed_bound_alpha}. 

\begin{flushleft}
    \textit{Step 1.}
\end{flushleft}
We consider the cases $p = 0$ and $p \neq 0$ separately. When $p=0$, if $\sS \cap \{z_1,\cdots,z_{2n}\} \neq \emptyset$, we obtain a factor $e^{-\frac{ {\theta_{\fy}}^2}{2\Lambda}}$ in the bound on coefficients as in Lemma~\ref{le:coefficients_bound_1}. If $\sS \cap \{z_1,\cdots,z_{2n}\} = \emptyset$, we apply Lemma~\ref{pr:no_singleton} control \eqref{e:aim_quantity} by the right hand side of \eqref{e:fixed_first_bound} since the sum of $X^{\diamond k}$ starts from $k=0$. This gives the bound \eqref{e:fixed_first_bound} when $p=0$. 

If $|p| \geq 1$, we have $|\theta_{2n+1}| \geq |\theta_{\fx}|$. Let $\set^*$ be the cluster containing $2n+1$. Then Lemma~\ref{le:coefficients_bound_2} applied to $\set^*$ gives us a decay factor $e^{-\frac{\theta_{\fx}^2}{4 \Lambda}}$ in the coefficient in chaos expansion. We can again choose $L_0$ large enough to obtain \eqref{e:fixed_first_bound}. 

\begin{flushleft}
    \textit{Step 2.}
\end{flushleft}
For any symmetric matrix $D=(d_{ij})_{0\leq i,j\leq 2n}$, we write
\begin{equation*}
    d_{i} := \sum_{j \neq i} d_{ij} = \sum_{j \neq i} d_{ji}\;.
\end{equation*}
Let $\dD$ and $\dD^*$ be the spaces of off-diagonal symmetric matrices with integer elements with the further restriction that
\begin{equation} \label{e:degree}
    0 \leq d_{0} \leq (m_1 \vee m_2)+1\;, \quad m_2 \leq d_{i} \leq (m_1 \vee m_2) + 1\;, \quad \forall i \geq 1 \;
\end{equation}
for $D \in \dD$, and
\begin{equation*}
    d_{0i}^{*} \equiv 0\;, \quad m_2 \leq d_{i}^* \leq (m_1 \vee m_2) + 1\;, \quad \forall i \geq 1
\end{equation*}
for $D^* \in \dD^*$, where we use $(d_{ij}^*)$ to denote elements in the matrix $\dD^*$, and $d_{i}^*$ is defined in the same way. Note that $\dD^* \subset \dD$ in further requiring that $d^*_{0i} = 0$. For $D \in \dD$, we define
\begin{equation*}
    W_D:= \bigg(\prod_{j=1}^{2n} (\E XY_{j})^{d_{0j}}\bigg) \bigg(\prod_{1\leq i<j \leq 2n} (\E Y_{i}Y_{j})^{d_{ij}}\bigg).
\end{equation*}
By Wick's formula, the left hand side of \eqref{e:fixed_first_bound} can be written as
\begin{equation} \label{e:Wick}
    \E \bigg[ \Big( \sum_{k=0}^{(m_1\vee m_2)+1}X^{\diamond k} \Big) \Big( \prod_{i=1}^{2n}  \sum_{k=m_2}^{(m_1\vee m_2)+1}Y_i^{\diamond k}\Big) \bigg] = \sum_{D \in \dD} W_D.
\end{equation}
We now show that for each $D \in \dD$, there exists $D^* \in \dD^*$ such that $W_{D} \lesssim \eps^{-\alpha((m_1\vee m_2)+1)} W_{D^*}$. This will imply that $W_{D}$ is bounded by the right hand side of \eqref{e:fixed_bound_alpha} and hence conclude the proof. 

Fix $D \in \dD$. If there exist $i,j$ with $1\leq i<j \leq 2n$ such that $d_{0i},d_{0j}>0$, then we perform the operation $D \mapsto D'$ by
\begin{equation*}
    d_{0i} \mapsto d_{0i}-1\;, \quad d_{0j} \mapsto d_{0j}-1\;, \quad d_{ij} \mapsto d_{ij}+2\;,
\end{equation*}
and the same operation to $d_{i0}, d_{j0}$ and $d_{ji}$. This operation keeps $D' \in \dD$. By the correlation bound (\cite[Proposition~2.1]{Xu18})
\begin{equation*}
    \E (XY_{i})\; \E (XY_{j}) \lesssim \E (Y_{i}Y_{j}), 
\end{equation*}
we have $W_{D} \lesssim W_{D'}$. We repeat this operation until there is at most one $i \in \{1, \dots, 2n\}$ such that $d_{0i}>0$. Thus, for any $D \in \dD$, there exists a symmetric matrix $D^{(1)} \in \dD$ with $W_{D} \lesssim W_{D^{(1)}}$, and there exists $i^* \in \{1, \dots, 2n\}$ such that $d^{(1)}_{0j} = 0$ for all $j \neq i^*$. We assume without loss of generality that this $i^*=1$. 

Starting with $D^{(1)}$, we perform the operation $D^{(1)} \mapsto D^{(2)}$ by setting
\begin{equation*}
    d^{(1)}_{01} \mapsto d^{(2)}_{01} = 0\;, \qquad d^{(1)}_{10} \mapsto d^{(2)}_{10} = 0\;. 
\end{equation*}
This removes the factor $(\E XY_1)^{d_{01}} \leq \Lambda^{d_{01}}$ from $W_{D^{(1)}}$ and hence gives the bound
\begin{equation*}
    W_{D^{(1)}} \lesssim W_{D^{(2)}}. 
\end{equation*}
Note that $D^{(2)}$ now satisfies $d_{0i}^{(2)} = 0$ for all $i$, $d^{(2)}_{ij} = d_{ij}$ for all $i,j \neq 1$, and $d^{(2)}_{1} \leq (m_1 \vee m_2) + 1$. If $d^{(2)}_{1} \geq m_2$, then $D^{(2)} \in \dD^*$ and hence we take $D^* = D^{(2)}$ to be the desired element in $\dD^*$. 

We now consider the case $d_{1}^{(2)} < m_2$, where $D^{(2)}$ does not belong to $\dD^*$. In this case, there exist $i,j \neq 1$ with $i \neq j$ such that $d_{ij} > 0$. Hence, we perform the operation $D^{(2)} \mapsto \tilde{D}$ by
\begin{equation*}
    d^{(2)}_{1i} \mapsto \tilde{d}_{1i} = d^{(2)}_{1i}+1\;, \quad d^{(2)}_{1j} \mapsto \tilde{d}_{1j} = d^{(2)}_{1j}+1\;, \quad d^{(2)}_{ij} \mapsto \tilde{d}_{ij} = d^{(2)}_{ij} - 1\;, 
\end{equation*}
and the same operation for $d^{(2)}_{i1}$, $d^{(2)}_{j1}$ and $d^{(2)}_{ji}$. This operation adds $d^{(2)}_{1}$ by $2$ and leaves $d^{(2)}_{j}$ unchanged for other $j$. Furthermore, it changes $W_{D^{(2)}}$ by a factor
\begin{equation*}
    \frac{(\E Y_1 Y_i) \; (\E Y_1 Y_j)}{\E Y_i Y_j} \gtrsim  \eps^{2\alpha}\;,
\end{equation*}
which gives the bound
\begin{equation*}
    W_{D^{(2)}} \lesssim \eps^{-2\alpha} W_{\tilde{D}}\;.
\end{equation*}
Hence, performing this operation at most $\frac{m_2 + 1}{2}$ times will give a symmetric matrix $D^{(3)}$ such that $d^{(3)}_{1} \in \{m_2, m_2 + 1\}$ and
\begin{equation*}
    W_{D^{(2)}} \lesssim \eps^{-\alpha(m_2+1)} W_{D^{(3)}}\;.
\end{equation*}
It is straightforward to check that $D^{(3)} \in \dD^*$. Hence, we take $D^* = D^{(3)}$ in this case and conclude the proof. 
\end{proof}
\endappendix

\bibliographystyle{Martin}
\bibliography{Refs}

\end{document}